\documentclass[a4paper,11pt,reqno]{amsart}

\usepackage{amsmath}
\usepackage{amsfonts}
\usepackage{amssymb}
\usepackage{amsthm}
\usepackage[shortlabels]{enumitem}
\usepackage{graphicx}
\usepackage{color}
\usepackage{dsfont}
\usepackage[latin1]{inputenc}
\usepackage{MnSymbol}
\usepackage{enumitem}
\usepackage{extpfeil}
\usepackage{url}
\usepackage[margin=1in]{geometry}
\usepackage{fancyhdr}
\usepackage[all]{xy}

\numberwithin{equation}{section}

\title{Finite presentations for K\"ahler groups with arbitrary finiteness properties}
\author{Claudio Llosa Isenrich}
\address{Mathematical Institute, Andrew Wiles Building, University of Oxford, Oxford OX2 6GG, UK}

\email{llosaisenric@maths.ox.ac.uk}
\thanks{This work was supported by a EPSRC Research Studentship and by the German National Academic Foundation}

\begin{document}

\newcommand{\QQ}{{\mathds Q}}
\newcommand{\RR}{{\mathds R}}
\newcommand{\NN}{{\mathds N}}
\newcommand{\ZZ}{{\mathds Z}}
\newcommand{\CC}{{\mathds C}}
\newcommand{\eps}{{\epsilon}}

\theoremstyle{plain}
\newtheorem{theorem}{Theorem}[section]
\newtheorem{conjecture}[theorem]{Conjecture}
\newtheorem{corollary}[theorem]{Corollary}
\newtheorem{lemma}[theorem]{Lemma}
\newtheorem{proposition}[theorem]{Proposition}
\newtheorem{question}{Question}

\theoremstyle{definition}
\newtheorem{remark}[theorem]{Remark}
\newtheorem*{acknowledgements*}{Acknowledgements}
\newtheorem{example}[theorem]{Example}

\renewcommand{\proofname}{Proof}

\begin{abstract}
 We construct the first explicit finite presentations for a family of K\"ahler groups with arbitrary finiteness properties, answering a question of Suciu.
\end{abstract}

\maketitle

\section{Introduction}
A \textit{K\"ahler group} is a finitely presented group that can be realised as fundamental group of a compact K\"ahler manifold. The difficult question of determining which finitely presented groups are K\"ahler groups was raised by Serre in the 1950s. While this has been a topic of active research for many years now, and a lot of very interesting results have been proved, we are still remarkably far from answering the question in a satisfactory way. For a comprehensive introduction to K\"ahler groups see \cite{ABCKT-95}. In this work we address the finiteness properties of K\"ahler groups.

A group $G$ is said to be of \textit{finiteness type $\mathcal{F}_r$} if it has a $K(G,1)$ with finite $r$-skeleton. It was shown by Dimca, Papadima, and Suciu \cite{DimPapSuc-09-II} that for every $r\geq 3$ there is a K\"ahler group of type $\mathcal{F}_{r-1}$ and not of type $\mathcal{F}_r$. These were the first examples of K\"ahler groups with arbitrary finiteness properties.

Suciu subsequently asked if it was possible to construct explicit presentations of such groups. In this work we will construct an explicit presentation for the examples in \cite{DimPapSuc-09-II}, thus answering Suciu's question. Our main result is

\begin{theorem}
\label{thmFinPres}
For each $r\geq 3$, the following finitely presented group $K_r$ is the fundamental group of a compact K\"ahler manifold; this group is of type $\mathcal{F}_{r-1}$, but not of type $\mathcal{F}_r$.
 \[
  K_r=\left\langle \left.   
  \mathcal{X}^{(r)}
  \right|
  \mathcal{R}^{(r)}_1, \mathcal{R}^{(r)}_2
  \right\rangle
 \]
where, for $r\geq 4$,
\[
 \mathcal{X}^{(r)}=\left\{ \begin{array}{l} 
 c_i,d,f_i^{(k)},g_i^{(k)},\\
 k=2,\cdots,r,i=1,2
 \end{array}\right\},
\]
\[
 \mathcal{R}_1^{(r)}=
 \left\{
 \begin{array}{l}
 (f_i^{(k)})^{\eps}c_j(f_i^{(k)})^{-\eps}(f_i^{(l)})^{\eps}(c_j)^{-1}(f_i^{(l)})^{-\eps},\\ (f_i^{(k)})^{\eps}d(f_i^{(k)})^{-\eps}(f_i^{(l)})^{\eps}d^{-1}(f_i^{(l)})^{-\eps},\left[\left[f_1^{(r)},f_2^{(r)}\right], d\right],\\
   \left[c_i,g_{j}^{(k)}\right] \left[(c_j)^{-1}f_j^{(k)},c_i\right],~ \left[f_i^{(k)},f_j^{(l)}\right]V_{i,j,1}^{-1},~\left[f_i^{(k)},g_j^{(l)}\right]W_{i,j,1}^{-1},\\  \left[c_ig_i^{(k)},c_jg_j^{(l)}\right] V_{i,j,1}^{-1},
   i,j=1,2, \eps=\pm 1, k,l=2,\cdots, r, l\neq k
 \end{array}
 \right\},
\]
\[
 \mathcal{R}_2^{(r)}=\left\{d^{-1}c_1^{-1}f_1^{(k)}c_2^{-1}(f_1^{(k)})^{-1}d^{-1}f_2^{(k)}c_1 (f_2^{(k)})^{-1}c_2,~ S^{(k)}\cdot T^{(k)}, k=2,\cdots,r \right\},
\]
where $V_{i,j,1}(A)$, $W_{i,j,1}(A)$, $i,j=1,2$, $S^{(k)}(A)$ and  $T^{(k)}(A)$ are words in the free group $F(A)$ in the generators $A=\left\{c_i,d,f_i^{(k)},g_i^{(k)},i=1,2,k=2,\cdots,r-1\right\}$ which will be defined in due course.

And $\mathcal{X}^{(3)}$, $\mathcal{R}^{(3)}_1$ and $\mathcal{R}^{(3)}_2$ are as described in Theorem \ref{thmFinPresV1}.
\end{theorem}

To be more precise, the group $K_r$ is a subdirect product of $r$ surface groups, the relations $\mathcal{R}_1$ correspond to the fact that elements of different factors in a direct product of groups commute and the relations $\mathcal{R}_2$ correspond to the surface group relations in the factors. 

To prove this result, we will apply algorithms developed by Baumslag, Bridson, Miller and Short \cite{BauBriMilSho-00} and by Bridson, Howie, Miller and Short \cite{BriHowMilSho-13} to construct explicit finite presentations for the groups of Dimca, Papadima, and Suciu \cite{DimPapSuc-09-II}. This leads to the explicit presentations given in Theorem \ref{thmFinPresV1}. We will then show by computations with Tietze transformations that these presentations simplify into the form of Theorem \ref{thmFinPres}.

 As a direct consequence of Theorem \ref{thmFinPres} we obtain
\begin{corollary}
 For $r\geq 3$, the first Betti number of the K\"ahler group $K_r$ is $b_1(K_r)=4r-2$.
\end{corollary}
\begin{proof}
 From the relations in the presentation in Theorem \ref{thmFinPres}, the definition of $V_{i,j,1}$ and $W_{i,j,1}$ in Lemma \ref{lemVW}, and the definition of $S^{(k)}$ and $T^{(k)}$ after Lemma \ref{lemCommRel2}, it is not hard to see that the only non-trivial relation in the canonical presentation for $$H_1(K_r,\ZZ)=\pi_1K_r/\left[\pi_1K_r,\pi_1K_r\right],$$ besides the commutator relations between the generators, is $d=1$. Hence, $H_1(K_r,\ZZ)$ is freely generated as an abelian group by $c_i$, $f_i^{(k)}$, $g_i^{(k)}$, where $i=1,2$, $k=2,\cdots,r$. It follows that $H_1(K_r,\ZZ)\cong \ZZ^{4r-2}$ and $b_1(K_r)=4r-2$.
\end{proof}

In recent work the author of this paper generalised Dimca, Papadima and Suciu's construction \cite{Llo-16-II} and the techniques used here give explicit finite presentations for many of the resulting groups. 

More generally there has been an increasing interest in the connection between K\"ahler groups and subgroups of direct products of surface groups. Starting with the fundamental work by Delzant and Gromov \cite{DelGro-05} on cuts in K\"ahler groups, a close connection between K\"ahler groups acting on CAT(0) cube complexes and subgroups of direct products of surface groups has emerged. Roughly speaking, a K\"ahler group which admits sufficiently many actions on a CAT(0) cube complex is a subgroup of a direct product of surface groups. Important recent results in this direction are Py's work \cite{Py-13} in which he shows that any subgroup of a Coxeter group or a RAAG is virtually a subgroup of a direct product of surface groups and Delzant and Py's work \cite{DelPy-16} showing that if a K\"ahler group acts geometrically on a locally finite CAT(0) cube complex then it is virtually the direct product of a free abelian group and finitely many surface groups. These results have intensified the interest in understanding the K\"ahler subgroups of direct products of surface groups and we anticipate that the techniques developed in this paper to give explicit descriptions of such groups will be useful in this context. Similarly explicit constructions of presentations should also illuminate some of the examples of the K\"ahler groups constructed by Bogomolov and Katzarkov \cite{BogKat-98}.

This work is structured as follows: In Section \ref{secDPS} we give a brief explanation of the construction of Dimca, Papadima and Suciu's groups. The proof of Theorem \ref{thmFinPresV1} is spread over Sections 3 and 4. In Section \ref{secProof} we show how one can simplify this presentation for $r\geq 4$ and thereby prove Theorem \ref{thmFinPres}.

\begin{acknowledgements*}
 I am very grateful to my PhD advisor, Martin Bridson, for his great support, his useful comments and in particular for referring me to the algorithms in \cite{BauBriMilSho-00} and \cite{BriHowMilSho-13}, which form the base of this work.
 After my talk about a method for constructing K\"ahler groups from maps onto complex tori \cite{Llo-17} at the September 2015 conference ``Finiteness Conditions in Topology and Algebra'' in Belfast, Alex Suciu asked me if I can give explicit finite presentations for K\"ahler groups with arbitrary finiteness properties. I am very grateful to him for this stimulus.
\end{acknowledgements*}

\section{Dimca, Papadima, and Suciu's examples}

\label{secDPS}

In \cite{DimPapSuc-09-II} Dimca, Papadima and Suciu constructed a family of examples of K\"ahler groups; in fact \textit{projective groups}, i.e. fundamental groups of compact complex projective varieties. We will sketch their construction.

Let $E=\CC/\Lambda$, $\Lambda \cong \ZZ ^2$, be an elliptic curve and let $B= \left\{b_1,\cdots, b_{2g-2}\right\}\subset E$ be a finite subset. Then $E$ has a 2-fold branched covering $S_g\rightarrow E$ by a surface $S_g$ of genus $g$ with branching locus $B$ as indicated in Figure \ref{fig:BranchedCover}.

\begin{figure}[ht] 
 \includegraphics[width=12cm,height=18cm,keepaspectratio]{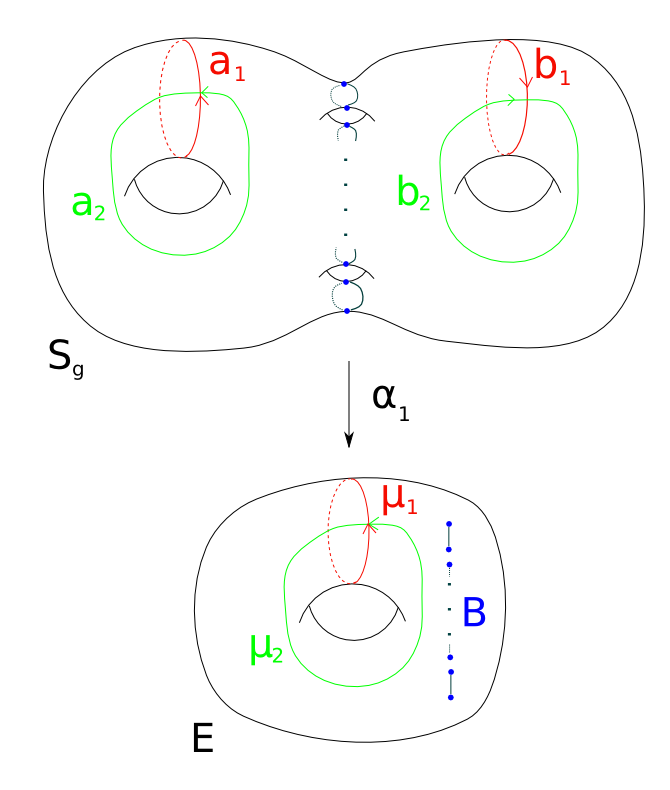}
 \caption{The 2-fold branched covering of $E$}
 \label{fig:BranchedCover}
\end{figure}

Let $\alpha_i:S^{(i)}=S_{g_i}^{(i)}\rightarrow E$, $i=1,\cdots,r$ be 2-fold branched covers of this form. Using addition in $E$ we can define the proper map
\[
 f=\sum_{i=1}^r \alpha_i: S^{(1)}\times \cdots \times S^{(r)}\rightarrow E.
\]

This map is a submersion with connected fibres away from a finite subset $C\subset S^{(1)}\times \cdots \times S^{(r)}$ with $f(C)=B^{(1)}\times\cdots\times B^{(r)}$. Hence, by the Ehresmann fibration theorem, all of its \textit{generic smooth fibres} over the open subset $E\setminus f(C)$ of regular values are homeomorphic. We denote by $H$ the generic smooth fibre of $f$. Dimca, Papdima and Suciu \cite[Theorem A]{DimPapSuc-09-II} proved that $\pi_1 H$ satisfies

\begin{theorem}
\label{thmDPSA}
 For all $r\geq 3$, the group $\pi_1 H$ is projective and therefore K\"ahler. Furthermore, $\pi_1 H$ is of type $\mathcal{F}_{r-1}$ but not of type $\mathcal{F}_r$.
\end{theorem}

We will now show how to construct an explicit finite presentation for $\pi_1 H$ for all $r\geq 3$. \textit{To simplify our computations we will only consider the case where $|B^{(i)}|=2$ and thus $g_i=2$ for $i=1,\cdots,r$.} A finite presentation for the general case can be constructed using the very same methods, but the ideas would be obscured by unnecessary complexity.

In our situation the fundamental group of $\pi_1 S^{(i)}$ has a finite presentation
\[
 \pi_1 S^{(i)} = \left\langle \left. a_1^{(i)},a_2^{(i)},b_1^{(i)},b_2^{(i)} ~ \right| \left[a_1^{(i)},a_2^{(i)}\right]\cdot\left[b_1^{(i)},b_2^{(i)}\right] \right\rangle
\]
where $a_j^{(i)},b_j^{(i)}$ are as indicated in Figure \ref{fig:BranchedCover} with an appropriate choice of base point. We see that with respect to the finite presentation
\[
 \pi_1 E = \left\langle \left. \mu_1,\mu_2~ \right\mid \left[\mu_1,\mu_2\right] \right\rangle
\]
the induced map $\alpha_{i*}$ on fundamental groups is given by
\[
 \begin{split}
   \alpha_{i*}: \pi_1 S^{(i)}& \longrightarrow \pi_1 E\\
   a^{(i)}_j,b^{(i)}_j&\longmapsto \mu_j, ~j=1,2.
 \end{split}
\]

In particular it follows that the induced map $f_*$ on fundamental groups is the map
\begin{equation}
 \begin{split}
  \phi=f_*: G=\pi_1S^{(1)}\times \cdots \times \pi_1S^{(r)} &\longrightarrow \pi_1 E\\
  a^{(i)}_j,b^{(i)}_j&\longmapsto \mu_j, ~i=1,\cdots,r,~j=1,2.
 \end{split}
 \label{eqngphom}
\end{equation}
We will prove that the groups corresponding to the finite presentation in Theorem \ref{thmFinPres} are isomorphic to $K= \mathrm{ker}\phi$. Hence they do indeed satisfy all of the stated properties by Theorem \ref{thmDPSA}.

Note also that $G$ has a presentation of the form
\begin{equation}
 G=\left\langle a_i^{(k)},b_i^{(k)},i=1,2,k=1,\cdots,r\mid \left[a_1^{(k)},a_2^{(k)}\right]\cdot\left[b_1^{(k)},b_2^{(k)} \right], \left[*^{(k)},*^{(l)}\right], l\neq k, k,l=1,\cdots,r  \right\rangle;
\label{eqnPresG}
 \end{equation}
 Here $*^{(k)}$ runs over all elements of the form $a_i^{(k)}, b_i^{(k)}$, for $i=1,2$.

\section{Some preliminary results}
\label{secPreLem}

With the examples in hand we do now proceed to prove a few preliminary results which will allow us to focus on the actual construction when proving Theorem \ref{thmFinPres}. Some of the results in this section will be fairly technical.

We will follow the definition of \textit{words} given in \cite[Section 1.3]{BauBriMilSho-00}: A \textit{word} $w(A)$ is a function that assigns to an ordered alphabet $A$ a word in the letters of $A \cup A^{-1}$ where $A^{-1}$ denotes the set of formal inverses of elements of $A$. This allows us to change between alphabets where needed. For example if $A=\left\{a,b\right\}$, $w(A)=aba$ and $A'=\left\{a',b'\right\}$, then $w(A')=a'b'a'$.

For a word $w(A)=a_1\cdots a_N$, with $a_1,\cdots, a_N\in A$, $N\in \NN$, we will denote by $\overline{w}(A)$ the word $a_N\cdots a_1$ and  denote by $w^{-1}(A)$ the word $w^{-1}(A)=a_N^{-1}\cdots a_1^{-1}$.

Following \cite[Section 1.4]{BauBriMilSho-00} we want to derive from the short exact sequence 
\[
 1\rightarrow K\rightarrow G\xrightarrow{\phi} \pi_1E\rightarrow 1
\]
a finite presentation for $G$ of the form 
\[
 \left\langle \mathcal{X}\cup \mathcal{A} \mid S_1 \cup S_2 \cup S_3\right\rangle,
\]
where $\mathcal{X}=\left\{x_1=a_1^{(1)},x_2=a_2^{(2)}\right\}$ is a lift of the generating set $\left\{ \mu_1, \mu_2\right\}$ of $\pi_1E=\left\langle \mu_1,\mu_2\mid \left[\mu_1,\mu_2\right]\right\rangle$ under the homomorphism $\phi:G\rightarrow \pi_1E$ and $\mathcal{A}=\left\{\alpha_1,\cdots, \alpha _n\right\}$ is a finite generating set of $K$. The relations are as follows:
\begin{itemize}
 \item $S_1$ contains a relation $x_i^{\eps}\alpha_j x_i^{-\eps}\omega_{i,j,\eps}(A)$ for every $i=1,2$, $j=1,\cdots,n$ and $\eps =\pm 1$, where $\omega_{i,j,\eps}(\mathcal{A})\in K$ is a word in $\mathcal{A}$ which is equal to $x_i^{\eps}\alpha_j^{-1}x_i^{-\eps}$ in $G$
 \item $S_2$ consists of one single relation of the form $\left[x_1,x_2\right] U(\mathcal{A})$ where $U(\mathcal{A})$ is an element of $K$ which is equal to $\left[x_1,x_2\right]^{-1}$ in $G$
 \item $S_3$ consists of a finite set of words in $\mathcal{A}$
\end{itemize}

We start by deriving a finite generating set for $K=\mathrm{ker}\phi$. One could do this by following the proof of the asymmetric 0-1-2 Lemma (cf. \cite[Lemma 1.3]{Bri-99} and \cite[Lemma 2.1]{BriGru-04}), but this would be no shorter than the more specific derivation followed here, which is more instructive.

\begin{proposition}
 For all $r\geq 2$, the group $K\leq \pi_1 S^{(1)}\times \cdots \times \pi _1 S^{(r)}=G$ is finitely generated with
 \[
  K= \left\langle \mathcal{K} \right\rangle,
 \]
 where $\mathcal{K}=\left\{c_1^{(1)},c_2^{(1)},d,f_i^{(k)},g_i^{(k)}, i=1,2,k=1,\cdots,r\right\}$ with the identifications \\
 $c_i^{(1)}=a_i^{(1)}(b_i^{(1)})^{-1}$, $d=\left[ b_1^{(1)},b_2^{(1)} \right]$, $f_i^{(k)}=a_i^{(1)} (a_i^{(k)})^{-1}$ and $g_i^{(k)}=b_i^{(1)}(b_i^{(k)})^{-1}$, $i=1,2$, $k=1,\cdots,r$.
 \label{propgenset}
\end{proposition}

 To ease notation we introduce the ordered sets
 \[
  X^{(k)}=\left\{a_1^{(k)},a_2^{(k)},b_1^{(k)},b_2^{(k)}\right\},~k=2,\cdots,r,
 \]
 \[
  Y^{(k)}=\left\{f_1^{(k)},f_2^{(k)},g_1^{(k)},g_2^{(k)}\right\}, ~k=2,\cdots, r
 \]
 
 The proof of Proposition \ref{propgenset} will make use of
\begin{lemma}
 Let $\mathcal{K}=\left\{c_1^{(1)},c_2^{(1)},d,f_i^{(k)},g_i^{(k)}, i=1,2,k=1,\cdots,r\right\}$ be as defined in Proposition \ref{propgenset}. Let $w(X^{(1)}\cup \cdots \cup X^{(m)})$ be a word in $X^{(1)}\cup \cdots \cup X^{(m)}$ with $m\in \left\{1,\cdots, r-1\right\}$ and let $v(X^{(1)})$ be a word in $X^{(1)}$. Then the following hold:
  \begin{enumerate}
  \item In $G$ we have the identity
  \[
   v(X^{(1)})\cdot w(X^{(1)}\cup\cdots \cup X^{(m)})\cdot v^{-1}(X^{(1)}) = v(Y^{(k)}) w(X^{(1)}\cup \cdots \cup X^{(m)}) v^{-1}(Y^{(k)}).
  \]
  In particular if $w(X^{(1)}\cup\cdots \cup X^{(m)})\in \left\langle \mathcal{K}\right\rangle$, then so are all its $\left\langle X^{(1)}\right\rangle$-conjugates.
  \item If $m=1$, we can cyclically permute the letters of $w(X^{(1)})$ using conjugation by elements in $\left\langle \mathcal{K} \right\rangle$.
  \item If $m=1$, all commutators of letters in $X^{(1)}$ are contained in $\left\langle \mathcal{K} \right\rangle$. 
  \item If $m=1$, we have $\phi(w(X^{(1)}))=1$ if and only if the combined sum of the exponents of $a_i^{(1)}$ and $b_i^{(1)}$ is zero for both, $i=1$ and $i=2$.
 \end{enumerate}
 \label{lemgenset}
\end{lemma}
\begin{proof}
 We obtain (1) using that $\left[X^{(k)},X^{(l)}\right]=\left\{1\right\}$ for $1\leq k\neq l\leq r$:
 {\small \[
  \begin{split}
  (a_i^{(1)})^{\eps}\cdot w(X^{(1)}\cup \cdots \cup X^{(m)})\cdot(a_i^{(1)})^{-\eps} &= ((a_i^{(1)})^{\eps} (a_i^{(k)})^{-\eps})\cdot w(X^{(1)}\cup \cdots \cup X^{(m)})\cdot ((a_i^{(k)})^{\eps}(a_i^{(1)})^{-\eps})\\
  &=(f_i^{(k)})^{\eps} w(X^{(1)}\cup \cdots \cup X^{(m)})(f_i^{(k)})^{-\eps},
  \end{split}
 \]
 \[
 \begin{split}
  (b_i^{(1)})^{\eps}\cdot w(X^{(1)}\cup \cdots \cup X^{(m)})\cdot(b_i^{(1)})^{-\eps} &= ((b_i^{(1)})^{\eps} (b_i^{(k)})^{-\eps})\cdot w(X^{(1)}\cup \cdots \cup X^{(m)})\cdot ((b_i^{(k)})^{\eps}(b_i^{(1)})^{-\eps})\\
  &=(g_i^{(k)})^{\eps} w(X^{(1)}\cup \cdots \cup X^{(m)})(g_i^{(k)})^{-\eps},\\
 \end{split}
 \]}
 for all $i=1,2$, $\eps = \pm 1$, $k > m$.
 
 We obtain (2) from (1). For instance $a_1^{(1)}w'(X^{(1)})= f_1^{(k)} w'(X^{(1)}) a_1^{(1)} (f_1^{(k)})^{-1}$.
 
 We obtain (3) from (1), (2) and the following identities in $G$
 \begin{itemize}
  \item $\left[b_1^{(1)},b_1^{(2)}\right]=d$,
  \item $\left[a_1^{(1)},a_2^{(1)}\right]=\left[b_1^{(1)},b_2^{(1)}\right]^{-1}$ in $\pi_1 S^{(1)}$ and thus in $G$,
  \item {\small $\left[a_i^{(1)},b_i^{(1)}\right]=\left[f_i^{(k)},(c_i^{(1)})^{-1}\right]$}, for $i=1,2$ and $k=2,\cdots, r$,
  \item $\left[a_i^{(1)},b_j^{(1)}\right]= c_i^{(1)}\cdot  d^{2j-3}\cdot g_j^{(k)} \cdot (c_1^{(1)})^{-1} \cdot (g_2^{(k)})^{-1}$ for and $i,j=1,2$, $k=2,\cdots,r$ and $i\neq j$.
 \end{itemize}
  With these commutators at hand we can use conjugation in $\left\langle \mathcal{K} \right\rangle$, cyclic permutation and inversion in order to obtain all other commutators.
  
  We obtain (4) as an immediate consequence of the definition of the map $\phi$ given in \eqref{eqngphom}.
\end{proof}

\begin{proof}[Proof of Proposition \ref{propgenset}]
 It is immediate from the explicit form \eqref{eqngphom} of the map $\phi$ that all elements in $\mathcal{K}$ are indeed contained in $K$. Hence, we only need to prove that these elements actually generate $K$.
 
 Let $g\in K$ be an arbitrary element. Since $\left[X^{(k)},X^{(l)}\right]=\left\{1\right\}$ in G, there are words $w_1(X^{(1)})$, $\cdots$, $w_r(X^{(r)})$ such that
 \[
  g= w_1(X^{(1)})\cdot~ \cdots ~ \cdot w_r(X^{(r)}).
 \]
 
 Using $\left[X^{(1)},X^{(l)}\right]=\left\{1\right\}$ for $l\neq 1$ we obtain that $w_k(X^{(k)})= \overline{w_k}(X^{(1)})\overline{w}_k^{-1}(Y^{(k)})$, $k=2,\cdots, r$ and consequently 
 \[
  \begin{split}
   g&=w_1(X^{(1)})\cdot~ \cdots ~ \cdot w_r(X^{(r)})\\
   &=w_1(X^{(1)})\cdot w_2(X^{(2)})\cdot ~ \cdots ~ \cdot w_{r-1}(X^{(r-1)})\cdot \overline{w_r}(X^{(1)})\cdot \overline{w_r}^{-1}(Y^{(r)})\\
   &= w_1(X^{(1)})\cdot \overline{w_r}(X^{(1)})\cdot w_2(X^{(2)})\cdot ~ \cdots ~ \cdot w_{r-1}(X^{(r-1)})\cdot  \overline{w_r}^{-1}(Y^{(r)})\\
   &= \cdots\\
   &= w_1(X^{(1)})\cdot \overline{w_r}(X^{(1)}) \overline{w_{r-1}}(X^{(1)})\cdot ~ \cdots ~ \cdot  \overline{w_2}(X^{(1)})\cdot  \overline{w_2}^{-1}(Y^{(2)})\cdot ~ \cdots ~ \cdot \overline{w}_{r-1}^{-1}(Y^{(r-1)})\cdot  \overline{w_r}^{-1}(Y^{(r)})
  \end{split}
 \]
 
 Since $g\in K$ and $\overline{w_2}^{-1}(Y^{(2)})\cdot ~ \cdots ~ \cdot  \overline{w_r}^{-1}(Y^{(r)})\in\left\langle \mathcal{K} \right\rangle \leq K$, it suffices to prove that every element of $K$ which is equal in $G$ to a word $w(X^{(1)})$ is in $\left\langle \mathcal{K} \right\rangle $.
 
 Due to Lemma \ref{lemgenset}(2),(3) we can use conjugation and commutators to obtain an equality
 \[
  w(X^{(1)})= u(\mathcal{K})\cdot (a_1^{(1)}) ^{m_1} \cdot (b_1^{(1)})^{n_1} \cdot (a_2^{(1)}) ^{m_2} \cdot (b_2^{(1)})^{n_2}\cdot v(\mathcal{K})
 \]
 in $K$ for some $m_1,m_2,n_1,n_2\in \ZZ$ and words $u(\mathcal{K}),v(\mathcal{K})\in \left\langle\mathcal{K}\right\rangle$. 
 
 By Lemma \ref{lemgenset}(4) we obtain that $n_1=-m_1$ and $n_2=-m_2$. Hence, another application of Lemma \ref{lemgenset}(2),(3) implies that
 \[
  w(X^{(1)})= u'(\mathcal{K})\cdot (c_1^{(1)})^{m_1} \cdot (c_2^{(1)}) ^{m_2}\cdot v'(\mathcal{K})\in \left\langle\mathcal{K}\right\rangle
 \]
 in $K$ for some words $u'(\mathcal{K}),v'(\mathcal{K})\in \left\langle\mathcal{K}\right\rangle$. This completes the proof. 
\end{proof}
 
We will use $\mathcal{K}$ as our generating set $\mathcal{A}$ and compute the elements of $S_1$ with respect to $\mathcal{K}$. 
 
 \begin{lemma}
  The following identities hold in $G$ for all $i,j=1,2$, $k=2,\cdots,r$ and $\eps=\pm 1$:

  {\fontsize{0.365cm}{1em}\selectfont
  \begin{flalign}
   \left[(a_i^{(1)})^{\eps},f_j^{(k)}\right]&=\left[(a_i^{(1)})^{\eps},a_j^{(1)}\right]=
   \left\{
   \begin{array}{ll}
    1 & \mbox{, if }i=j\\
    d^{2i-3}& \mbox{, if }i\neq j\mbox{ and } \eps=1\\
    (f_i^{(k)})^{-1}d^{3-2i}f_i^{(k)}& \mbox{, if } i\neq j\mbox{ and } \eps =-1.
   \end{array}
   \right.&&
   \label{eqncomm1}
  \end{flalign}
  \begin{flalign}
   \left[(a_i^{(1)})^{\eps},g_j^{(k)}\right]=\left[(a_i^{(1)})^{\eps},b_j^{(1)}\right]&=
   \left\{
   \begin{array}{ll}
    \left[(f_i^{(k)})^{\eps},(c_i^{(1)})^{-1}\right] & \mbox{, if }i=j\\
    f_i^{(k)} (c_j^{(1)})^{-1} (f_i^{(k)})^{-1}d^{2i-3}c_j^{(1)}& \mbox{, if }i\neq j\mbox{ and } \eps=1\\
    (f_i^{(k)})^{-1} (c_j^{(1)})^{-1}d^{3-2i} f_i^{(k)}c_j^{(1)}& \mbox{, if } i\neq j\mbox{ and } \eps =-1.
   \end{array}
   \right.&&
   \label{eqncomm2}
  \end{flalign}
  \begin{flalign}
  \label{eqnnorm1}
   (a_i^{(1)})^{\eps}c_j^{(1)}(a_i^{(1)})^{-\eps}&=(f_i^{(k)})^{\eps}c_j^{(1)}(f_i^{(k)})^{-\eps} &&
  \end{flalign}
  \begin{flalign}
  \label{eqnnorm2}
   (a_i^{(1)})^{\eps}d(a_i^{(1)})^{-\eps} &=(f_i^{(k)})^{\eps}d(f_i^{(k)})^{-\eps} &&
  \end{flalign}
}
\label{lemVW}
 \end{lemma}
 
 \begin{proof}
  The vanishing $\left[X^{(l)},X^{(k)}\right]=\left\{1\right\}$ for $k\neq l$ yields the equalities
  \[
   \left[(a_i^{(1)})^{\eps},f_j^{(k)}\right]=\left[(a_i^{(1)})^{\eps},a_j^{(1)}\right],
  \]
  \[
   \left[(a_i^{(1)})^{\eps},g_j^{(k)}\right]=\left[(a_i^{(1)})^{\eps},b_j^{(1)}\right],
  \]
  as well as the equalities \eqref{eqnnorm1} and \eqref{eqnnorm2}. In particular the commutators on the left of \eqref{eqncomm1} and \eqref{eqncomm2} are independent of $k$.
  
  The equalities on the right of \eqref{eqncomm1} and \eqref{eqncomm2} are established similarly.

 \end{proof}
 
 In the following we will denote by $V_{i,j,\eps}(\mathcal{K})$ the words in the alphabet $\mathcal{K}$ as defined on the right side of equation \eqref{eqncomm1} and by $W_{i,j,\eps}(\mathcal{K})$ the words in the alphabet $\mathcal{K}$ as defined on the right side of equation \eqref{eqncomm2}, in both cases choosing $k=2$. 
 
 With this notation we obtain
 \[
  S_1=\left\{
    \left[x_i^{\eps},f_j^{(k)}\right]V_{i,j,\eps}^{-1},~ \left[x_i^{\eps},g_j^{(k)}\right]W_{i,j,\eps}^{-1},~ x_i^{\eps}c_jx_i^{-\eps}(f_i^{(k)})^{\eps}c_j^{-1}(f_i^{(k)})^{-\eps},~ x_i^{\eps}dx_i^{-\eps}(f_i^{(k)})^{\eps}d^{-1}(f_i^{(k)})^{-\eps}
  \right\}.
 \]
 In fact we do not actually need all of the relations $S_1$, but we are able to express some of them in terms of the other ones
 \begin{lemma}
  There is a canonical isomorphism
  \[
   \left\langle \mathcal{X},\mathcal{K}\mid S_1'\right\rangle \cong \left\langle \mathcal{X},\mathcal{K}\mid S_1\right\rangle,
  \]
  induced by the identity map on generators, with
  \[
   S_1'=\left\{
    \left[x_i,f_j^{(k)}\right]V_{i,j,1}^{-1},~ \left[x_i,g_j^{(k)}\right]W_{i,j,1}^{-1},~ x_i^{\eps}c_jx_i^{-\eps}(f_i^{(k)})^{\eps}c_j^{-1}(f_i^{(k)})^{-\eps},~ x_i^{\eps}dx_i^{-\eps}(f_i^{(k)})^{\eps}d^{-1}(f_i^{(k)})^{-\eps}
   \right\}.
  \]
  \label{lemNoEps}
 \end{lemma}

 \begin{proof}
  Since the set of relations $S_1'$ is a proper subset of $S_1$ it suffices to prove that all elements in $S_1 \setminus S_1'$ are products of conjugates of relations in $S_1'$. Indeed, we have the following equalities in $\left\langle \mathcal{X},\mathcal{K}\mid S_1'\right\rangle$ using relations of the form \eqref{eqnnorm1} and \eqref{eqnnorm2} and the relations for $\eps=1$:
  \[
   \begin{split}
   \left[x_i^{-1},f_j^{(k)}\right]V_{i,j,-1}^{-1}&= x_i^{-1}f_j^{(k)}x_i(f_j^{(k)})^{-1} V_{i,j,-1}^{-1}\\
   &=x_i^{-1}\left[x_i,f_j^{(k)}\right]^{-1} x_i V_{i,j,-1}^{-1}\\
   &=\left\{\begin{array}{ll}x_i^{-1} x_i&, \mbox{ if } i=j\\ x_i^{-1} d^{3-2i} x_i (f_i^{(k)})^{-1} d^{2i-3} f_i^{(k)} &, \mbox{ if } i\neq j   \end{array} \right.\\
   &=1
   \end{split}
  \]
  and
  \[
   \begin{split}
   \left[x_i^{-1},g_j^{(k)}\right]W_{i,j,-1}^{-1}&= x_i^{-1}g_j^{(k)}x_i(g_j^{(k)})^{-1} W_{i,j,-1}^{-1}\\
   &=x_i^{-1}\left[x_i,g_j^{(k)}\right]^{-1} x_i W_{i,j,-1}^{-1}\\
   &=\left\{\begin{array}{ll}x_i^{-1} \left[(c_i^{(1)})^{-1},f_i^{(k)}\right] x_i \left[(c_i^{(1)})^{-1},(f_i^{(k)})^{-1}\right]&, \mbox{ if } i=j\\ 
   x_i^{-1} (c_j^{(1)})^{-1}d^{3-2i}f_i^{(k)}c_j^{(1)}(f_i^{(k)})^{-1} x_i (c_j^{(1)})^{-1} (f_i^{(k)})^{-1} d^{2i-3} c_j^{(1)} f_i^{(k)} &, \mbox{ if } i\neq j   \end{array} \right.\\
   &= \left\{\begin{array}{ll}x_i^{-1} \left[(c_i^{(1)})^{-1},f_i^{(k)}\right]  \left[x_i (c_i^{(1)})^{-1}x_i^{-1},x_i (f_i^{(k)})^{-1}x_i^{-1}\right]x_i&, \mbox{ if } i=j\\
   x_i^{-1} (c_j^{(1)})^{-1}d^{3-2i}f_i^{(k)}c_j^{(1)}(f_i^{(k)})^{-1}f_i^{(k)} (c_j^{(1)})^{-1}(f_i^{(k)})^{-1}\cdot &, \mbox{ if } i\neq j\\   (f_i^{(k)})^{-1} f_i^{(k)} d^{2i-3} (f_i^{(k)})^{-1} f_i^{(k)} c_j^{(1)} (f_i^{(k)})^{-1} f_i^{(k)}x_i &    \end{array} \right.\\
   & = \left\{\begin{array}{ll}x_i^{-1} \left[(c_i^{(1)})^{-1},f_i^{(k)}\right]  \left[f_i^{(k)} (c_i^{(1)})^{-1}(f_i^{(k)})^{-1},(f_i^{(k)})^{-1}\right]x_i&, \mbox{ if } i=j\\ 
   1 &, \mbox{ if } i\neq j   \end{array} \right.\\
   & = 1 \\
   \end{split}
  \]

 \end{proof}

 To obtain $S_2$ observe that using $\left[X^{(1)},X^{(k)}\right]=\left\{1\right\}$ for $k\neq 1$ and the relation $\left[a_1^{(1)},a_2^{(1)}\right]\cdot d$ in $G$, we obtain
 \[
   \left[x_1,x_2\right] = d^{-1}
 \]
 
 Hence, the set of relations $S_2$ is given by
 \[
  S_2=\left\{\left[x_1,x_2\right] d \right\}.
 \]
 
 To obtain the set $S_3$ we recall that $G$ has a finite presentation of the form \eqref{eqnPresG}. Hence, it suffices to express all of the relations in the presentation \eqref{eqnPresG} as words in $\mathcal{K}$ modulo relations of the form $S_1$ and $S_2$. For group elements $g,h$ we write $g\sim h$ if $g$ and $h$ are in the same conjugacy class.
 
 \begin{lemma}
  In the free group $F(\mathcal{X}\cup \mathcal{K})$ modulo the relations $S_1$ and $S_2$, and the identifications made in Proposition \ref{propgenset}, we obtain the following equivalences of words:
 \begin{enumerate}
  \item $\left[a_i^{(1)},a_j^{(k)}\right] \sim 1$
  \item $\left[a_i^{(1)},b_j^{(k)}\right] \sim 1$
  \item $\left[b_i^{(1)},a_j^{(k)}\right] \sim 1$
 \end{enumerate}
 for all $i,j=1,2$ and $k=2,\cdots, r$.
 \label{lemCommRel1}
 \end{lemma}
 
 \begin{proof} ~
 
 (1) follows from $\left[x_i,f_j^{(k)}\right] V_{i,j,1}^{-1}$:
  \[
  \begin{split}
  \left[a_i^{(1)}, a_j^{(k)}\right] &= \left[x_i,(f_j^{(k)})^{-1}x_j\right] \sim \left[x_i,x_j\right]\left[f_j^{(k)},x_i\right]=1 
  \end{split}
  \] 
 
 (2) follows from $\left[x_i,g_j^{(k)}\right]W_{i,j,1}^{-1}$:
 \[
   \begin{split}
  \left[a_i^{(1)}, b_j^{(k)}\right] &= \left[x_i,(g_j^{(k)})^{-1}(c_j^{(1)})^{-1}x_j\right] \sim \left[g_j^{(k)},x_i\right] x_i (c_j^{(1)})^{-1} x_i^{-1} \left[x_i,x_j\right] c_j^{(1)} = 1 
  \end{split} 
 \]
 
 (3) follows from $\left[x_i,f_j^{(k)}\right] V_{i,j,1}^{-1}$ and $x_jc_i^{(1)}x_j^{-1}f_j^{(k)}(c_i^{(1)})^{-1}(f_j^{(k)})^{-1}$:
 \[
  \begin{split}
  \left[b_i^{(1)}, a_j^{(k)}\right] &= \left[(c_i^{(1)})^{-1}x_i,(f_j^{(k)})^{-1}x_j\right] \sim x_jc_i^{(1)}x_j^{-1} f_j^{(k)}(c_i^{(1)})^{-1}(f_j^{(k)})^{-1}\left[f_j^{(k)},x_i\right]\left[x_i,x_j\right] = 1 
  \end{split} 
 \]
\end{proof}

 \begin{lemma}
  In the free group $F(\mathcal{X}\cup \mathcal{K})$ modulo the relations $S_1$, $S_2$, the identifications made in Proposition \ref{propgenset}, and the relations (1)-(3) from Lemma \ref{lemCommRel1}, we obtain the following equivalences of words:
  \begin{enumerate}
   \item $\left[b_i^{(1)},b_j^{(k)}\right] \sim \left[c_i^{(1)},g_j^{(k)}\right]\left[(c_j^{(1)})^{-1} f_j^{(k)}, c_i^{(1)}\right]$
  \item $\left[a_i^{(k)},a_j^{(l)}\right]=(f_i^{(k)})^{-1}(f_j^{(l)})^{-1}V_{i,j,1}^{-1}f_i^{(k)}f_j^{(l)}  \sim \left[f_i^{(k)},f_j^{(l)}\right]V_{i,j,1}^{-1}$
  \item $\left[a_i^{(k)},b_j^{(l)}\right] \sim \left[f_i^{(k)},g_j^{(l)}\right] W_{i,j,1}^{-1}$
  \item $\left[b_i^{(k)},b_j^{(l)}\right] =(c_i^{(1)}g_i^{(k)})^{-1}(g_j^{(l)}c_j^{(1)})^{-1} V_{i,j,1}^{-1} (c_i^{(1)}g_i^{(k)})(c_j^{(1)}g_j^{(k)})
 \sim \left[(c_i^{(1)}g_i^{(k)}),(c_j^{(1)}g_j^{(k)})\right] V_{i,j,1}^{-1}$
 \end{enumerate}
  for all $i,j=1,2$ and $k,l=2,\cdots, r$.
  \label{lemCommRel2}
 \end{lemma}
\begin{proof}
 (1) follows from Lemma \ref{lemCommRel1}(2) and the relation $x_jc_i^{(1)}x_j^{-1}f_j^{(k)}(c_i^{(1)})^{-1}(f_j^{(k)})^{-1}$: 
 \[
   \begin{split}
  \left[b_i^{(1)}, b_j^{(k)}\right] &= b_i^{(1)}(a_i^{(1)})^{-1}b_j^{(k)}(b_j^{(1)})^{-1} b_j^{(1)} a_i^{(1)}(b_i^{(1)})^{-1} (b_j^{(1)})^{-1}b_j^{(1)}(b_j^{(k)})^{-1}\\
  &\sim \left[c_i^{(1)},g_j^{(k)}\right]b_j^{(1)}c_i^{(1)}(b_j^{(1)})^{-1} (c_i^{(1)})^{-1}\\
  &=\left[c_i^{(1)},g_j^{(k)}\right]b_j^{(1)}(a_j^{(1)})^{-1} a_j^{(1)} c_i^{(1)} (a_j^{(1)})^{-1} a_j^{(1)} (b_j^{(1)})^{-1} (c_i^{(1)})^{-1}\\
  &=\left[c_i^{(1)},g_j^{(k)}\right]\left[(c_j^{(1)})^{-1} f_j^{(k)}, c_i^{(1)}\right]
  \end{split} 
 \]
 
 (2) follows from Lemma \ref{lemCommRel1}(1):
 \[
  \begin{split}
  \left[a_i^{(k)}, a_j^{(l)}\right] &= a_i^{(k)} (a_i^{(1)})^{-1} a_j^{(l)}(a_j^{(1)})^{-1} \left[a_j^{(1)},a_i^{(1)}\right] a_i^{(1)} (a_i^{(k)})^{-1} a_j^{(1)} (a_j^{(k)})^{-1}\\
  &= (f_i^{(k)})^{-1} (f_j^{(l)})^{-1} V_{i,j,1}^{-1} f_i^{(k)} f_j^{(l)}\\
  &\sim \left[f_i^{(k)},f_j^{(l)}\right] V_{i,j,1}^{-1}
  \end{split} 
 \]
 
 (3) follows from Lemma \ref{lemCommRel1}(2),(3) and (4) follows from Lemma \ref{lemCommRel1}(2) by similar calculations.

\end{proof}

 We introduce the words $$S^{(k)}(\mathcal{K})=(f_1^{(k)})^{-1}(f_2^{(k)})^{-1}df_1^{(k)}f_2^{(k)}$$ and $$T^{(k)}(\mathcal{K})=(c_1^{(1)}g_1^{(k)})^{-1} (c_2^{(1)} g_2^{(k)})^{-1}d c_1^{(k)}g_1^{(k)}c_2^{(1)}g_2^{(k)}$$ for $2\leq k\leq r$. Notice that these words appear in the presentation in Theorem \ref{thmFinPres}.
 
 The only relation of $G$ that we did not express, yet, is the relation $\left[a_1^{(1)},a_2^{(2)}\right]\left[b_1^{(1)},b_2^{(2)}\right]$. Modulo $S_1$ and $S_2$ it satisfies
  \[
  \begin{split}
   \left[a_1^{(1)},a_2^{(1)}\right]\left[b_1^{(1)},b_2^{(1)}\right] & = \left[x_1,x_2\right]\left[(c_1^{(1)})^{-1}x_1,(c_2^{(1)})^{-1}x_2\right]\\
   & = d^{-1} (c_1^{(1)})^{-1} x_1 (c_2^{(1)})^{-1}x_2 x_1^{-1} c_1^{(1)} x_2^{-1} c_2^{(1)}\\
   & = d^{-1} (c_1^{(1)})^{-1} f_1^{(k)} (c_2^{(1)})^{-1} (f_1^{(k)})^{-1} d^{-1} f_2^{(k)} c_1^{(1)} (f_2^{(k)})^{-1}c_2^{(1)}
  \end{split}
  \]

 We define the set of relations $S_3$ by
 \[
  S_3=\left\{ 
  \begin{array}{l}
    \left[f_i^{(k)},f_j^{(l)}\right]V_{i,j,1}^{-1}, \left[f_i^{(k)},g_j^{(l)}\right] W_{i,j,1}^{-1},\\
    \left[c_i^{(1)},g_j^{(k)}\right] W_{i,j,1}^{-1}c_i^{(1)} V_{i,j,1}^{-1}(c_i^{(1)})^{-1},    \left[g_i^{(k)},c_jg_j^{(l)}\right]W_{j,i,1},\\
    d^{-1}c_1^{-1}f_1^{(k)}c_2^{-1}(f_1^{(k)})^{-1}d^{-1}f_2^{(k)}c_1 (f_2^{(k)})^{-1}c_2, S^{(k)}\cdot T^{(k)}    
  \end{array}
  \right\}
 \]

 \begin{theorem}
  The group $G=\pi_1S^{(1)}\times \cdots \times \pi_1 S^{(r)}$ is isomorphic to the group defined by the finite presentation $\left\langle \mathcal{X}, \mathcal{K}\mid \mathcal{R}\right\rangle=$ 
 {\small \[
\sbox0{$
  \begin{array}{l} 
  x_1,x_2, c_1^{(1)},c_2^{(1)},d,\\ 
  f_i^{(k)},g_i^{(k)},\\
   k=2,\cdots,r,\\i=1,2
  \end{array}
  \left|
  \begin{array}{l}
   \left[x_i,f_j^{(k)}\right]V_{i,j,1}^{-1},~ \left[x_i,g_j^{(k)}\right]W_{i,j,1}^{-1},~ x_i^{\eps}(c_j^{(1)})x_i^{-\eps}(f_i^{(k)})^{\eps}(c_j^{(1)})^{-1}(f_i^{(k)})^{-\eps},\\
   ~ x_i^{\eps}dx_i^{-\eps}(f_i^{(k)})^{\eps}d^{-1}(f_i^{(k)})^{-\eps}, \left[x_1,x_2\right] \cdot d,\\
   \left[c_i^{(1)},g_{j}^{(k)}\right] \left[(c_j^{(1)})^{-1}f_j^{(k)},c_i^{(1)}\right],~ \left[f_i^{(k)},f_j^{(l)}\right]V_{i,j,1}^{-1},~\left[f_i^{(k)},g_j^{(l)}\right]W_{i,j,1}^{-1},\\  \left[c_i^{(1)}g_i^{(k)},c_j^{(1)}g_j^{(l)}\right] V_{i,j,1}^{-1},\\
   d^{-1}(c_1^{(1)})^{-1}f_1^{(k)}(c_2^{(1)})^{-1}(f_1^{(k)})^{-1}d^{-1}f_2^{(k)}c_1^{(1)} (f_2^{(k)})^{-1}c_2^{(1)},~ S^{(k)}\cdot T^{(k)},\\
   i,j=1,2, \eps=\pm 1, k,l=2,\cdots, r, l\neq k
  \end{array}
  \right.
  $}
\mathopen{\resizebox{1.2\width}{\ht0}{$\Bigg\langle$}}
\usebox{0}
\mathclose{\resizebox{1.2\width}{\ht0}{$\Bigg\rangle$}}
\]
}
 \label{thmFinPresA}
 \end{theorem}
 
 \begin{proof}
  By construction there is a canonical way of identifying $G$ given by the presentation \ref{eqnPresG} with the presentation $\left\langle \mathcal{X}, \mathcal{K}\mid \mathcal{R}\right\rangle$. Namely, consider the map on generators $\mathcal{X}\cup \mathcal{K}$ defined by:
  \[
   \begin{split}
   x_i & \mapsto a_i^{(1)}\\
   c_i^{(1)}&\mapsto a_i^{(1)}(b_i^{(1)})^{-1}\\
   d &\mapsto \left[b_1^{(1)},b_2^{(1)}\right]\\
   f_i^{(k)} &\mapsto a_i^{(1)} (a_i^{(k)})^{-1}\\
   g_i^{(k)} &\mapsto b_i^{(1)} (b_i^{(k)})^{-1}\\
   \end{split}
  \]
  By construction of $S_1'$, $S_2$ and $S_3$ the image of all relations in $\mathcal{R}$ vanishes in $G$. In particular this map extends to a well-defined group homomorphism $\psi: \left\langle \mathcal{X}, \mathcal{K}\mid \mathcal{R}\right\rangle\rightarrow G$. The map $\psi$ is onto, because of the identities $a_i^{(1)}=\phi(x_i)$, $b_i^{(1)}=\phi(c_i^{(1)})^{-1}\phi(x_i)$, $a_i^{(k)}=\phi(f_i^{(k)})^{-1} \phi(x_i)$, $b_i^{(k)}= \phi(g_i^{(k)})^{-1}b_i^{(1)}$. 
  
  Hence, we only need to check that $\phi$ is injective. For this it suffices to construct a well-defined inverse homomorphism. It is obtained by considering the map on generators of $G$ defined by
  \[
   \begin{split}
    a_i^{(1)} & \mapsto x_i\\
    b_i^{(1)} & \mapsto (c_i^{(1)})^{-1} x_i\\
    a_i^{(k)} & \mapsto (f_i^{(k)})^{-1} x_i,~k\geq 2\\
    b_i^{(k)} & \mapsto (g_i^{(k)})^{-1} (c_i^{(1)})^{-1} x_i, ~k\geq 2\\
   \end{split}
  \]
  
  Since we expressed all relations in the presentation \ref{eqnPresG} in terms of the generators of $\mathcal{R}$ under the identification given by $\psi$ using only relations of the form $S_1$ and $S_2$, they vanish trivially under this map and thus there is an extension to a group homomorphism $\psi ^{-1}: G\rightarrow \left\langle \mathcal{X}, \mathcal{K}\mid \mathcal{R}\right\rangle$ inverse to $\psi$.
  
  For instance
  \[
  \begin{split}
   \psi^{-1}\left(\left[a_1^{(1)},a_2^{(1)}\right]\left[b_1^{(1)},b_2^{(1)}\right]\right) & = \left[x_1,x_2\right]\left[(c_1^{(1)})^{-1}x_1,(c_2^{(1)})^{-1}x_2\right]\\
   & = d^{-1} (c_1^{(1)})^{-1} x_1 (c_2^{(1)})^{-1}x_2 x_1^{-1} c_1^{(1)} x_2^{-1} c_2^{(1)}\\
   & = d^{-1} (c_1^{(1)})^{-1} f_1^{(k)} (c_2^{(1)})^{-1} (f_1^{(k)})^{-1} d^{-1} f_2^{(k)} c_1^{(1)} (f_2^{(k)})^{-1}c_2^{(1)}
  \end{split}
  \]
  which is indeed a relation in $\mathcal{R}$. Similarly we obtain that $\psi^{-1}$ vanishes on all other relations by going through the proofs of Lemma \ref{lemCommRel1} and \ref{lemCommRel2}.
  
 \end{proof}
 
 We will now deduce a presentation for the group $\pi_1 S^{(r)}=\left\langle \left[a_1^{(r)},a_2^{(r)}\right]\left[b_1^{(r)},b_2^{(r)}\right]\right\rangle$ of the form of Remark 2.1(1) in \cite{BriHowMilSho-13} with respect to the epimorphism $\pi_1 S^{(r)}\rightarrow \pi_1 E$, $a_i^{(r)},b_i^{(r)}\mapsto \mu_i$ and the presentation $\left\langle \mu_1^{-1},\mu_2^{-1}\mid \left[\mu_1^{-1},\mu_2^{-1}\right]\right\rangle$ of $\pi_1E$. That is, we derive a presentation of the form $\left\langle \left\{x_1,x_2\right\},C\mid \widehat{R}, S\right\rangle$ such that $x_i\mapsto \mu_i^{-1}$, $c\mapsto 1$, $\widehat{R}$ consists of a relation of the form $\left[x_1,x_2\right]U(C)$ and $S$ consists of a finite set of words in $C^*$. Here $C^*$ is defined to be the set of conjugates of elements of $C$ by words in the free group on $X$.
 
 \begin{proposition}
  The finite presentation
  \[
   \left\langle x_1,x_2,c_1^{(r)},c_2^{(r)},\delta \mid \left[x_1,x_2\right]\delta, \delta ^{-2} x_2x_1\left[(x_1c_1^{(r)})^{-1},(x_2c_2^{(r)})^{-1}\right]x_2^{-1}x_1^{-1} \right\rangle
  \]

  is a presentation for $\pi_1S^{(r)}$ of the form described in the previous paragraph, with the isomorphism given by $x_i\mapsto (a_i^{(r)})^{-1}$, $c_1^{(r)}\mapsto a_i^{(r)}(b_i^{(r)})^{-1}$ and $\delta \mapsto \left[(a_2^{(r)})^{-1},(a_1^{(r)})^{-1}\right]$.
 \label{propFinPresB}
 \end{proposition}
 \begin{proof}
  Using Tietze transformations and the identifications $x_i=(a_i^{(r)})^{-1}$, $c_1^{(r)}=a_i^{(r)}(b_i^{(r)})^{-1}$, $\delta = \left[(a_2^{(r)})^{-1},(a_1^{(r)})^{-1}\right]$, we obtain
  \[
   \begin{split}
   \pi_1S^{(r)}&=\left\langle a_1^{(r)},a_2^{(r)},b_1^{(r)},b_2^{(r)}\mid \left[a_1^{(r)},a_2^{(r)}\right]\left[b_1^{(r)},b_2^{(r)}\right]\right\rangle\\
   &= \left\langle x_1,x_2,c_1^{(r)},c_2^{(r)},\delta \mid \left[x_1,x_2\right]\delta, \left[x_1^{-1},x_2^{-1}\right]\left[(x_1c_1^{(r)})^{-1},(x_2c_2^{(r)})^{-1}\right] \right\rangle
   \end{split}
 \]
 Using Tietze transformations and the relation $\left[x_1,x_2\right]\delta$ we obtain
\[
 \begin{split}
  &\left[x_1^{-1},x_2^{-1}\right]\left[(x_1c_1^{(r)})^{-1},(x_2c_2^{(r)})^{-1}\right]\\
  &=x_2^{-1}x_1^{-1}\delta ^{-2}  (x_2 x_1 (c_1^{(r)})^{-1} (x_2x_1)^{-1}) (x_2 (c_2^{(r)})^{-1} x_2^{-1}) (x_1 c_1^{(r)} x_1^{-1})  (x_1x_2c_2^{(r)}(x_1x_2)^{-1}) x_1x_2\\
  &\sim \delta ^{-2} \cdot (x_2 x_1 (c_1^{(r)})^{-1} (x_2x_1)^{-1})\cdot (x_2 (c_2^{(r)})^{-1} x_2^{-1}) \cdot (x_1 c_1^{(r)} x_1^{-1}) \cdot (x_1x_2c_2^{(r)}(x_1x_2)^{-1} )
 \end{split}
 \]
 This completes the proof.
 \end{proof}
 
 A subgroup $H\leq \Gamma_1\times \cdots \times \Gamma_r$ of a direct product is called \textit{subdirect} if its projection to every factor is surjective.
 
 \begin{lemma}
  The subgroup $K\leq \pi_1S^{(1)}\times\cdots\times \pi_1 S^{(r)}$ is subdirect and in fact the projection
  \[
  q: K\longrightarrow \pi_1S^{(1)}\times \cdots \times \pi_1 S^{(r-1)}
  \]
  is surjective.
  \label{lemProj}
 \end{lemma}
 
 \begin{proof}
  It suffices to prove the second part of the assertion, since $K$ is symmetric in the factors. To see that the projection $K\rightarrow \pi_1S^{(1)}\times \cdots \times \pi_1 S^{(r-1)}$ is surjective, observe that $f_i^{(r)}\mapsto a_i^{(1)}$ and $g_i^{(r)} \mapsto b_i^{(1)}$ under the projection map. Hence, we obtain that $(f_i^{(k)})^{-1}f_i^{(r)}\mapsto a_i^{(k)}$ and $(g_i^{(k)})^{-1} g_i^{(r)} \mapsto b_i^{(k)}$ under the projection map.
 \end{proof}

\section{Construction of a Presentation}
\label{secProofV1}
We will now follow the algorithms described in \cite[Theorem 3.7]{BriHowMilSho-13} and \cite[Theorem 2.2]{BriHowMilSho-13} in order to derive a finite presentation for $K$. We start by proving that we obtain the following presentation for $K$:
\begin{theorem}
\label{thmFinPresV1}
 Let $r\geq 3$. Then the group defined by the finite presentation
  {\small \[
\sbox0{$
  \begin{array}{l} 
  x_1,x_2, f_1^{(r)},f_2^{(r)},\\ 
  A=\left\{c_1^{(1)},c_2^{(1)},d, f_i^{(k)},g_i^{(k)},\right.\\
   k=2,\cdots,r-1,i=1,2\Bigr\} \\
  B=\left\{c_1^{(r)},c_2^{(r)}, \delta \right\}
  \end{array}
  \left|
  \begin{array}{l}
   x_i^{\eps}c_j^{(1)}x_i^{-\eps}(f_i^{(k)})^{\eps}(c_j^{(1)})^{-1}(f_i^{(k)})^{-\eps},~ x_i^{\eps}dx_i^{-\eps}(f_i^{(k)})^{\eps}d^{-1}(f_i^{(k)})^{-\eps},\\ \left[x_1,x_2\right]\cdot \delta \cdot d,\\
   \left[c_i^{(1)},g_{j}^{(k)}\right] \left[(c_j^{(1)})^{-1}f_j^{(k)},c_i^{(1)}\right],~ \left[f_i^{(k)},f_j^{(l)}\right]V_{i,j,1}^{-1},~\left[f_i^{(k)},g_j^{(l)}\right]W_{i,j,1}^{-1},\\  \left[c_i^{(1)}g_i^{(k)},c_j^{(1)}g_j^{(l)}\right] V_{i,j,1}^{-1},\\
   d^{-1}(c_1^{(1)})^{-1}f_1^{(k)}(c_2^{(1)})^{-1}(f_1^{(k)})^{-1}d^{-1}f_2^{(k)}c_1^{(1)} (f_2^{(k)})^{-1}c_2^{(1)},~ S^{(k)}\cdot T^{(k)},\\
   \delta ^{-2} x_2x_1\left[(x_1c_1^{(r)})^{-1},(x_2c_2^{(r)})^{-1}\right]x_2^{-1}x_1^{-1},~ \left[A,B\right],f_1^{(r)}=x_1,f_2^{(r)}=x_2\\ 
   i,j=1,2, \eps=\pm 1, k,l=2,\cdots, r-1, l\neq k
  \end{array}
  \right.
  $}
\mathopen{\resizebox{1.2\width}{\ht0}{$\Bigg\langle$}}
\usebox{0}
\mathclose{\resizebox{1.2\width}{\ht0}{$\Bigg\rangle$}}
\]}
is a K\"ahler group of finiteness type $\mathcal{F}_{r-1}$, but not of finiteness type $\mathcal{F}_r$, where $V_{i,j,\eps}(A)$, $W_{i,j,\eps}(A)$, $i,j=1,2,\eps=\pm 1$, $T^{(k)}$ and  $S^{(k)}$ are the words in the free group on $A$ defined above. 
\end{theorem}
For $r=3$ we define
\[
 \mathcal{X}^{(3)}=
 \left\{ x_1,x_2,A,B\right\}
\]
\[
 \mathcal{R}_1^{(3)}=\left\{\ 
 \begin{array}{l}                             
   x_i^{\eps}c_j^{(1)}x_i^{-\eps}(f_i^{(2)})^{\eps}(c_j^{(1)})^{-1}(f_i^{(2)})^{-\eps},~ x_i^{\eps}dx_i^{-\eps}(f_i^{(2)})^{\eps}d^{-1}(f_i^{(2)})^{-\eps}, \left[x_1,x_2\right]\cdot \delta \cdot d,\\
   \left[c_i^{(1)},g_{j}^{(2)}\right] \left[(c_j^{(1)})^{-1}f_j^{(2)},c_i^{(1)}\right],~ \left[f_i^{(3)},f_j^{(2)}\right]V_{i,j,1}^{-1},~\left[f_i^{(3)},g_j^{(2)}\right]W_{i,j,1}^{-1},\\ \left[A,B\right], i,j=1,2,~ \eps= \pm 1
  \end{array}
   \right\}
\]
\[
 \mathcal{R}_2^{(3)}=\left\{\ 
 \begin{array}{l}                             
   d^{-1}(c_1^{(1)})^{-1}f_1^{(2)}(c_2^{(1)})^{-1}(f_1^{(2)})^{-1}d^{-1}f_2^{(2)}c_1^{(1)} (f_2^{(2)})^{-1}c_2^{(1)}, S^{(2)}\cdot T^{(2)},\\
   \delta ^{-2} x_2x_1\left[(x_1c_1^{(3)})^{-1},(x_2c_2^{(3)})^{-1}\right]x_2^{-1}x_1^{-1}
  \end{array}
   \right\}
\]

\begin{proof}[Proof of Theorem \ref{thmFinPresV1}]
By Proposition \ref{propgenset} the set $\mathcal{K}=\left\{c_1^{(1)},c_2^{(1)},d,f_i^{(k)},g_i^{(k)}, i=1,2,k=1,\cdots,r\right\}$ is a generating set of $K$ where $c_i^{(1)}=a_i^{(1)}(b_i^{(1)})^{-1}$, $d=\left[ b_1^{(1)},b_2^{(1)} \right]$, $f_i^{(k)}=a_i^{(1)} (a_i^{(k)})^{-1}$ and $g_i^{(k)}=b_i^{(1)}(b_i^{(k)})^{-1}$. By Lemma \ref{lemProj} the projection $p_{ij}(\mathcal{K})$ to the group $\pi_1S^{(i)}\times \pi_1S^{(j)}$ is surjective, since $r\geq 3$ by assumption. 

Hence, Theorem 3.7 in \cite{BriHowMilSho-13} provides us with an algorithm that will output a finite presentation of $K$. Since by Lemma \ref{lemProj} the projection $\Gamma_2= q(K)= \pi_1 S^{(1)}\times \cdots \times \pi_1 S^{(r-1)}$ is surjective, Theorem \ref{thmFinPresA} provides us with a finite presentation for $\Gamma_2$ given by
{\small 
\begin{equation}
\sbox0{$
   \begin{array}{l} 
  x_1,x_2, c_1^{(1)},c_2^{(1)},d,\\ 
  f_i^{(k)},g_i^{(k)},\\
   k=2,\cdots,r-1,\\i=1,2
  \end{array}
  \left|
  \begin{array}{l}
   \left[x_i,f_j^{(k)}\right]V_{i,j,1}^{-1},~ \left[x_i,g_j^{(k)}\right]W_{i,j,1}^{-1},~ x_i^{\eps}c_j^{(1)}x_i^{-\eps}(f_i^{(k)})^{\eps}(c_j^{(1)})^{-1}(f_i^{(k)})^{-\eps},\\ x_i^{\eps}dx_i^{-\eps}(f_i^{(k)})^{\eps}d^{-1}(f_i^{(k)})^{-\eps},
   \left[x_1,x_2\right] \cdot d,\\
 \left[c_i^{(1)},g_{j}^{(k)}\right] \left[(c_j^{(1)})^{-1}f_j^{(k)},c_i^{(1)}\right],~ \left[f_i^{(k)},f_j^{(l)}\right]V_{i,j,1}^{-1},~\left[f_i^{(k)},g_j^{(l)}\right]W_{i,j,1}^{-1},\\  \left[c_i^{(1)}g_i^{(k)},c_j^{(1)}g_j^{(l)}\right] V_{i,j,1}^{-1},\\
   d^{-1}(c_1^{(1)})^{-1}f_1^{(k)}(c_2^{(1)})^{-1}(f_1^{(k)})^{-1}d^{-1}f_2^{(k)}c_1^{(1)} (f_2^{(k)})^{-1}c_2^{(1)},~ S^{(k)}\cdot T^{(k)},\\
   i,j=1,2, \eps=\pm 1, k,l=2,\cdots, r, l\neq k
  \end{array}
  \right.
  $}
\mathopen{\resizebox{1.2\width}{\ht0}{$\Bigg\langle$}}
\usebox{0}
\mathclose{\resizebox{1.2\width}{\ht0}{$\Bigg\rangle$}}
\label{eqnPresGamma2}
\end{equation}
}

For $\Gamma_1 =\pi_1S^{(r)}$ we choose the finite presentation derived in Proposition \ref{propFinPresB}
  \begin{equation}
   \Gamma_1 \cong \left\langle x_1,x_2,c_1^{(r)},c_2^{(r)},\delta \mid \left[x_1,x_2\right]\delta, \delta ^{-2} x_2x_1\left[(x_1c_1^{(r)})^{-1},(x_2c_2^{(r)})^{-1}\right]x_2^{-1}x_1^{-1} \right\rangle
  \label{eqnPresGamma1}
  \end{equation}
and for $Q=\Gamma_1/(\Gamma_1\cap K)\cong \pi_1 E$ we choose the finite presentation
\[
 Q\cong \left\langle x_1,x_2\mid \left[x_1,x_2\right]\right\rangle.
\]

We further introduce the notation
\[
 \mathcal{X}=\left\{x_1,x_2\right\},
\]
\[
 \mathcal{A} =\left\{c_1^{(1)},c_2^{(1)},d, f_i^{(k)},g_i^{(k)},k=2,\cdots,r-1,i=1,2\right\},
\]
\[
 \mathcal{B}=\left\{c_1^{(r)},c_2^{(r)},\delta\right\}.
\]

Then the canonical projections $f_1: \Gamma_1 \rightarrow Q$ and $f_2: \Gamma_2 \rightarrow Q$ are given by the identity on $\mathcal{X}$ and by mapping $\mathcal{A}$ and $\mathcal{B}$ to $1$. Note further that the presentation $\mathcal{Q}=\left\langle x_1,x_2\mid \left[x_1,x_2\right]\right\rangle$ of $Q$ satisfies $\pi_2\mathcal{Q}=1$.

By construction the group $K$ in Section \ref{secDPS} is then isomorphic to the fibre product of the projections $f_1:\Gamma_1\rightarrow Q$ and $f_2: \Gamma_2 \rightarrow Q$.

It follows that it suffices to apply the algorithm described in the proof of the Effective Asymmetric 1-2-3 Theorem \cite[Theorem 2.2]{BriHowMilSho-13} by Bridson, Howie, Miller and Short in order to obtain a finite presentation of $K$. We will follow their notation as far as possible. For a more detailed explanation of the following steps we recommend the reader to have a look at the original source \cite{BriHowMilSho-13}.

The first step of the algorithm considers the recursively enumerable class $\mathcal{C}(\mathcal{Q})$ of presentations of the form
\[
 \left\langle \mathcal{X}\cup A \cup B \mid S_1,S_2,S_3,S_4,S_5\right\rangle,
\]
with
\begin{itemize}
 \item $S_1$ consists of a single relation of the form $\left[x_1,x_2\right]u(A)v(B^*)$, where $u(A)$ is a word in the free group on the letters of $A$ and $v(B^*)$ is a word in the free group on the letters of $B^*$, the set of all formal conjugates of letters in $B$ by elements in the free group on $\mathcal{X}$,
 \item $S_2$ consists of a relator $x^{\eps}ax^{-\eps}\omega_{a,x,\eps}$ for every $a\in A$, $x\in X$ and $\eps = \pm 1$ with $\omega_{a,x,\eps}(A)$ a word in the free group on $A$,
 \item $S_3=\left\{ aba^{-1}b^{-1}\mid a\in A, b\in B\right\}$,
 \item $S_4$ is a finite set of words in the free group on $A$,
 \item $S_5$ is a finite set of words in the free group on $B$.
\end{itemize}

Denote by $H$ the group corresponding to this presentation, by $N_A$, respectively $N_B$, the normal closure of $A$, respectively $B$, in $H$ and by $H_A=H/N_A$, respectively $H_B=H/N_B$, the corresponding quotients. Note in particular that there is a canonical isomorphism $Q\cong H/(N_A\cdot N_B)$ and hence there are canonical quotient maps $\pi_A: H_A\mapsto Q$ and $\pi_B: H_B\mapsto Q$.

The algorithm runs through presentations in the class $\mathcal{C}(\mathcal{Q})$ and stops when it finds a presentation such that there are isomorphisms $\phi_A: H_A\rightarrow \Gamma_1$ and $\phi_B: H_B\rightarrow \Gamma_2$ with the property that $f_1\circ \phi_A=\pi_A$ and $f_2\circ \phi _B=\pi_B$.

By construction of the presentations \eqref{eqnPresGamma2} and \eqref{eqnPresGamma1} and Lemma \ref{lemNoEps} it follows that the presentation
 {\small \begin{equation}
\sbox0{$
  \begin{array}{l} 
  x_1,x_2,\\ 
  \mathcal{A}=\left\{c_1^{(1)},c_2^{(1)},d, f_i^{(k)},g_i^{(k)},\right.\\
   k=2,\cdots,r-1,i=1,2\Bigr\} \\
  \mathcal{B}=\left\{c_1^{(r)},c_2^{(r)}, \delta \right\}
  \end{array}
  \left|
  \begin{array}{l}
   \left[x_i^{\eps},f_j^{(l)}\right]V_{i,j,\eps}^{-1},~\left[x_i^{\eps},g_j^{(l)}\right]W_{i,j,\eps}^{-1},~ x_i^{\eps}c_j^{(1)}x_i^{-\eps}(f_i^{(k)})^{\eps}(c_j^{(1)})^{-1}(f_i^{(k)})^{-\eps},\\
    x_i^{\eps}dx_i^{-\eps}(f_i^{(k)})^{\eps}d^{-1}(f_i^{(k)})^{-\eps},\left[x_1,x_2\right]\cdot \delta \cdot d,\\
   \left[c_i^{(1)},g_{j}^{(k)}\right] \left[(c_j^{(1)})^{-1}f_j^{(k)},c_i^{(1)}\right],~ \left[f_i^{(k)},f_j^{(l)}\right]V_{i,j,1}^{-1},~\left[f_i^{(k)},g_j^{(l)}\right]W_{i,j,1}^{-1},\\  \left[c_i^{(1)}g_i^{(k)},c_j^{(1)}g_j^{(l)}\right] V_{i,j,1}^{-1},\\
   d^{-1}(c_1^{(1)})^{-1}f_1^{(k)}(c_2^{(1)})^{-1}(f_1^{(k)})^{-1}d^{-1}f_2^{(k)}c_1^{(1)} (f_2^{(k)})^{-1}c_2^{(1)},~ S^{(m)}\cdot T^{(m)},\\
   \delta ^{-2} x_2x_1\left[(x_1c_1^{(r)})^{-1},(x_2c_2^{(r)})^{-1}\right]x_2^{-1}x_1^{-1},~ \left[A,B\right]\\ 
   i,j=1,2, \eps=\pm 1, k=2,\cdots, r, l,m=2,\cdots, r-1, l,m\neq k
  \end{array}
  \right.
  $}
\mathopen{\resizebox{1.2\width}{\ht0}{$\Bigg\langle$}}
\usebox{0}
\mathclose{\resizebox{1.2\width}{\ht0}{$\Bigg\rangle$}}
\label{eqnPresKNo1}
\end{equation}}
is of this form with $f_1$ and $f_2$ defined by the identity maps on $\mathcal{X}$, $\mathcal{A}$ and $\mathcal{B}$. Again we denote by $H$ the group corresponding to this presentation.

The second part of the algorithm derives a finite generating set $Z$ for the $\ZZ Q$-module  $N_A\cap N_B$. The set $Z$ is obtained from an arbitrary choice, but fixed, finite choice of identity sequences $M$ that generate $\pi_2 \mathcal{Q}$ as a $\ZZ Q$-module. In particular the elements of $Z$ are in one-to-one correspondence with the elements of $M$. The details of the construction of the elements of $M$ from the elements of $Z$ can be found in the proof of \cite[Theorem 1.2]{BauBriMilSho-00}.

In our case we have $\pi_2 \mathcal{Q}=1$. Hence, we can choose $M=\emptyset$ implying that $Z=\emptyset$ and in particular $N_A\cap N_B= 1$. But the algorithm tells us that $K\cong H/(N_A\cap N_B) = H$. Thus, the algorithm shows that the presentation \eqref{eqnPresKNo1} is indeed a finite presentation for $K$ and by construction the isomorphism between $K$ and $H$ is induced by the map
\[
\begin{split}
 x_i& \mapsto f_i^{(r)}\\
 c_i^{(1)}&\mapsto c_i^{(1)}\\
 d &\mapsto d\\
 f_i^{(k)}&\mapsto f_i^{(k)}\\
 g_i^{(k)}&\mapsto g_i^{(k)}\\
 c_i^{(r)}&\mapsto (f_i^{(r)})^{-1}c_i^{(1)}g_i^{(r)}\\
 \delta &\mapsto \left[f_2^{(r)},f_1^{(r)}\right]d^{-1}\\
 \end{split}
\]
on generating sets.

Applying Lemma \ref{lemNoEps} in order to reduce the relations of the form $S_2$ to a subset of $S_2$ and introducing the generators $f_i^{(r)}=x_i$ completes the proof.

\end{proof}
\section{Proof of Theorem \ref{thmFinPres}}
\label{secProof}
We will now show how Theorem \ref{thmFinPres} follows from Theorem \ref{thmFinPresV1}. The proof follows easily from the following three lemmas whose proofs we will give at the end of this section.

\begin{lemma}
 Applying Tietze transformations to the presentation in Theorem \ref{thmFinPresV1}, we can replace the set of relations
 \[
  \left\{\begin{array}{l}
  \left[c_i^{(r)},c_j^{(1)}\right],\left[c_i^{(r)},f_j^{(k)}\right],\left[c_i^{(r)},g_j^{(k)}\right],\\
  \left[\delta,d\right], i=1,2,k=2,\cdots,r-1
  \end{array}
  \right\}\subseteq \left[A,B\right]
 \]
 by the set of relations
 \[
  \left\{\begin{array}{l}
  \left[c_i^{(1)},g_j^{(r)}\right]\left[(c_j^{(1)})^{-1}f_j^{(r)},c_i^{(1)}\right],\left[f_i^{(k)},g_j^{(r)}\right]W_{i,j,1}^{-1},\left[c_i^{(1)}g_i^{(r)},c_j^{(1)}g_j^{(k)}\right]V_{i,j,1}^{-1},\\
  \left[\left[x_2,x_1\right],d\right] i,j=1,2,k=2,\cdots,r-1
  \end{array}\right\}
 \]
 under the identifications $x_i=f_i^{(r)}$, $g_i^{(r)}=(c_i^{(1)})^{-1}f_i^{(r)}c_i^{(r)}$ and $\delta = \left[x_2,x_1\right]d^{-1}$.
 \label{lemProofMT1}
\end{lemma}

Denote by $\mathcal{M}$ the subset of the set of relations of the presentation for $K$ in Theorem \ref{thmFinPresV1} defined by 
 \[
  \mathcal{M}=\left\{\begin{array}{l}
     \left[d,c_i^{(r)}\right],\left[\delta,c_i^{(1)}\right],\left[\delta,f_i^{(k)}\right],\left[\delta,g_i^{(k)}\right],\\ i=1,2,k=1,\cdots,r-1
  \end{array}\right\},
 \]
and denote by $\mathcal{M}^C$ its complement in the set of all relations in the presentation for $K$ given in Theorem \ref{thmFinPresV1}.

\begin{lemma}
For $r\geq 4$, all elements of $\mathcal{M}$ can be expressed as product of conjugates of relations in its complement $\mathcal{M}^C$. Therefore we can remove the set $\mathcal{M}$ from the set of relations of $K$ using Tietze transformations.
 \label{lemProofMT2}
\end{lemma}

The third result we want to use is
\begin{lemma}
 In the presentation for $K$ given in Theorem \ref{thmFinPresV1} we can replace the relation $$x_2^{-1}x_1^{-1}\delta^{-2}x_2x_1\left[(x_1c_1^{(r)})^{-1},(x_2c_2^{(r)})^{-1}\right]=1$$ by the relation $$S^{(r)}T^{(r)}=1$$ using Tietze transformations and the identifications $x_i=f_i^{(r)}$, $g_i^{(r)}=(c_i^{(1)})^{-1}f_i^{(r)}c_i^{(r)}$ and $\delta = \left[x_2,x_1\right]d^{-1}$.
 \label{lemProofMT3}
\end{lemma}

We are now ready to prove Theorem \ref{thmFinPres}.

\begin{proof}[Proof of Theorem \ref{thmFinPres}]
 Start with the presentation for $K_r$ derived in Theorem \ref{thmFinPresV1} and use the identifications $x_i=f_i^{(r)}$, $g_i^{(r)}=(c_i^{(1)})^{-1}f_i^{(r)}c_i^{(r)}$ and $\delta = \left[x_2,x_1\right]d^{-1}$. 
 
 From Lemma \ref{lemProofMT3} we obtain that using Tietze transformations we can replace the relation $$\delta^{-2}x_2x_1\left[(x_1c_1^{(r)})^{-1},(x_2c_2^{(r)})^{-1}\right]x_2^{-1}x_1^{-1}$$ by the relation $$S^{(r)}T^{(r)}.$$
 
 Lemma \ref{lemProofMT2} implies that we can remove the relations
  \[
  \left\{\begin{array}{l}
     \left[d,c_i^{(1)}\right],\left[\delta,c_i^{(1)}\right],\left[\delta,f_i^{(k)}\right],\left[\delta,g_i^{(k)}\right],\\ i=1,2,k=1,\cdots,r-1
  \end{array}\right\},
 \]
 from our presentation.
 
 Lemma \ref{lemProofMT1} implies that we can replace the set of relations
  \[
  \left\{\begin{array}{l}
  \left[c_i^{(r)},c_j^{(1)}\right],\left[c_i^{(r)},f_j^{(k)}\right],\left[c_i^{(r)},g_j^{(k)}\right],\\
  \left[\delta,d\right], i=1,2,k=2,\cdots,r-1
  \end{array}
  \right\}\subseteq \left[A,B\right]
 \]
 by the set of relations
 \[
  \left\{\begin{array}{l}
  \left[c_i^{(1)},g_j^{(r)}\right]\left[(c_j^{(1)})^{-1}f_j^{(r)},c_i^{(1)}\right],\left[f_i^{(k)},g_j^{(r)}\right]W_{i,j,1}^{-1},\left[c_i^{(1)}g_i^{(r)},c_j^{(1)}g_j^{(k)}\right]V_{i,j,1}^{-1},\\
  \left[\left[f_2^{(r)},f_1^{(r)}\right],d\right] i,j=1,2,k=2,\cdots,r-1
  \end{array}\right\}.
 \]
 
 The identification $c_i^{(1)}=c_i$ thus shows that $K_r$ is isomorphic to the group with the finite presentation
 {\small \[K=K_r=
\sbox0{$
  \begin{array}{l} 
   c_i,d,f_i^{(k)},g_i^{(k)},\\
   k=2,\cdots,r,i=1,2\\
  \end{array}
  \left|
  \begin{array}{l}
   (f_i^{(k)})^{\eps}c_j(f_i^{(k)})^{-\eps}(f_i^{(l)})^{\eps}(c_j)^{-1}(f_i^{(l)})^{-\eps},\\ (f_i^{(k)})^{\eps}d(f_i^{(k)})^{-\eps}(f_i^{(l)})^{\eps}d^{-1}(f_i^{(l)})^{-\eps},\left[\left[f_1^{(r)},f_2^{(r)}\right], d\right],\\
   \left[c_i,g_{j}^{(k)}\right] \left[(c_j)^{-1}f_j^{(k)},c_i\right],~ \left[f_i^{(k)},f_j^{(l)}\right]V_{i,j,1}^{-1},~\left[f_i^{(k)},g_j^{(l)}\right]W_{i,j,1}^{-1},\\  \left[c_ig_i^{(k)},c_jg_j^{(l)}\right] V_{i,j,1}^{-1},\\
   d^{-1}c_1^{-1}f_1^{(k)}c_2^{-1}(f_1^{(k)})^{-1}d^{-1}f_2^{(k)}c_1 (f_2^{(k)})^{-1}c_2,~ S^{(k)}\cdot T^{(k)},\\
   i,j=1,2, \eps=\pm 1, k,l=2,\cdots, r, l\neq k
  \end{array}
  \right.
  $}
\mathopen{\resizebox{1.2\width}{\ht0}{$\Bigg\langle$}}
\usebox{0}
\mathclose{\resizebox{1.2\width}{\ht0}{$\Bigg\rangle$}}.
\]}
\end{proof}

We will now prove the three lemmas:

\begin{proof}[Proof of Lemma \ref{lemProofMT1}]
 To simplify notation we enumerate the relations as follows:
 \begin{enumerate}
 \item $\left[c_i^{(1)},g_j^{(r)}\right]\left[(c_j^{(1)})^{-1}f_j^{(r)},c_i^{(1)}\right]=1$
 \item $\left[f_i^{(k)},g_j^{(r)}\right]W_{i,j,1}^{-1}=1$
 \item $\left[c_i^{(1)}g_i^{(r)},c_j^{(1)}g_j^{(l)}\right]V_{i,j,1}^{-1}=1$
 \item $\left[\left[x_2,x_1\right],d\right]=1$
 \end{enumerate}
 
 The following computation shows that relation (1) is equivalent to $\left[c_i^{(r)},c_j^{(1)}\right]=1$:
\begin{flalign*}
  \left[c_i^{(1)},g_j^{(r)}\right]\left[(c_j^{(1)})^{-1}f_j^{(r)},c_i^{(1)}\right] &=c_i^{(1)}(c_j^{(1)})^{-1}f_j^{(r)}c_j^{(r)}(c_i^{(1)})^{-1}(c_j^{(r)})^{-1}(f_j^{(r)})^{-1}c_j^{(1)}\left[(c_j^{(1)})^{-1}f_j^{(r)},c_i^{(1)}\right]&\\
  &=c_i^{(1)}(c_j^{(1)})^{-1}f_j^{(r)}(c_i^{(1)})^{-1}c_j^{(r)}(c_j^{(r)})^{-1}(f_j^{(r)})^{-1}c_j^{(1)}\left[(c_j^{(1)})^{-1}f_j^{(r)},c_i^{(1)}\right]&=1 \\
 \end{flalign*}

Now we show that relation (2) is equivalent to the relation $\left[c_j^{(r)},f_i^{(k)}\right]=1$:
\begin{flalign*}
   \left[f_i^{(k)},g_j^{(r)}\right] W_{i,j,1}^{-1} &= f_i^{(k)}(c_j^{(1)})^{-1}f_j^{(r)}c_j^{(r)}(f_i^{(k)})^{-1}(c_j^{(r)})^{-1}(f_j^{(r)})^{-1}c_j^{(1)} W_{i,j,1}^{-1} &\\
   &=f_i^{(k)}(c_j^{(1)})^{-1}f_j^{(r)}(f_i^{(k)})^{-1}c_j^{(r)}(c_j^{(r)})^{-1}(f_j^{(r)})^{-1}c_j^{(1)} W_{i,j,1}^{-1} &\\
   &=\left[f_i^{(k)}, (c_j^{(1)})^{-1}x_j\right] W_{i,j,1}^{-1}=1
\end{flalign*}

Next we show that relation (3) is equivalent to $\left[c_j^{(r)},g_j^{(k)}\right]$ modulo all other relations:
\begin{flalign*}
    \left[c_i^{(1)}g_i^{(r)},c_j^{(1)}g_j^{(l)}\right] V_{i,j,1}^{-1}
    &= f_i^{(r)}c_i^{(r)}c_j^{(1)}g_j^{(l)}(c_i^{(r)})^{-1}(f_i^{(r)})^{-1} (g_j^{(l)})^{-1} (c_j^{(1)})^{-1} V_{i,j,1}^{-1}&\\
    &=f_i^{(r)}c_j^{(1)}g_j^{(l)}c_i^{(r)}(c_i^{(r)})^{-1}(f_i^{(r)})^{-1} (g_j^{(l)})^{-1} (c_j^{(1)})^{-1} V_{i,j,1}^{-1}&\\ 
    &=f_i^{(r)}c_j^{(1)}(f_i^{(r)})^{-1}W_{i,j,1}(c_j^{(1)})^{-1}V_{i,j,1}^{-1}=1 &
\end{flalign*}
Note that here the only relation from $\left[A,B\right]$ that we use is $\left[c_i^{(r)},c_j^{(1)}g_j^{(l)}\right]=1$ which modulo $\left[c_i^{(r)},c_j^{(1)}\right]=1$ is equivalent to $\left[c_i^{(r)},g_j^{(l)}\right]=1$. This means that modulo (1) the relation $\left[c_i^{(r)},g_j^{(l)}\right]=1$ is equivalent to the relation (3).

Equivalence of relation (4) and $\left[\delta,d\right]$ is immediate from $\delta=\left[x_2,x_1\right]d^{-1}$. 
\end{proof}

\begin{proof}[Proof of Lemma \ref{lemProofMT2}]

We need to show that the following relations follow from the relations in $\mathcal{M}^C$:
\begin{enumerate}
 \item $\left[c_i^{(r)},d\right]=1$
 \item $\left[\delta,c_i^{(1)}\right]=1$
 \item $\left[\delta,f_i^{(k)}\right]=1$
 \item $\left[\delta,g_i^{(k)}\right]=1$
\end{enumerate}

For (1), given $r\geq 4$ we can choose $2\leq l\neq k\leq r-1$ with
\[
\left[c_i^{(r)},d\right]=\left[c_i^{(r)},\left[c_1^{(1)}g_1^{(k)},c_2^{(1)}g_2^{(l)}\right]\right]=1.
\]

For (2), using $d=\left[f_2^{(r)},f_1^{(l)}\right]$ we obtain that it suffices to prove that
\[
 (c_i^{(1)})^{-1}f_2^{(r)}f_1^{(r)}(f_2^{(r)})^{-1}f_2^{(k)}(f_1^{(r)})^{-1}(f_2^{(k)})^{-1}c_i^{(1)}=f_2^{(r)}f_1^{(r)}(f_2^{(r)})^{-1}f_2^{(k)}(f_1^{(r)})^{-1}(f_2^{(k)})^{-1}
\]

This equality is a consequence of the relations $\left[f_i^{(k)},f_j^{(l)}\right]=V_{i,j,\epsilon}=d^{\epsilon}$, $f_2^{(k)}d(f_2^{(k)})^{-1}=f_2^{(r)}d(f_2^{(r)})^{-1}$, and $f_2^{(k)}c_i^{(1)}(f_2^{(k)})^{-1}=f_2^{(r)}c_i^{(1)}(f_2^{(r)})^{-1}$:
\[\begin{split}
 ~& (c_i^{(1)})^{-1}f_2^{(r)}f_1^{(r)}(f_2^{(r)})^{-1}f_2^{(k)}(f_1^{(r)})^{-1}(f_2^{(k)})^{-1}c_i^{(1)}\\
 =&f_2^{(r)}(f_2^{(l)})^{-1}(c_i^{(1)})^{-1}df_1^{(r)}f_2^{(l)}(f_2^{(r)})^{-1}f_2^{(k)}(f_2^{(l)})^{-1}(f_1^{(r)})^{-1}d^{-1}c_i^{(1)}f_2^{(l)}(f_2^{(r)})^{-1} \\
 =& f_2^{(r)}(f_2^{(l)})^{-1}f_1^{(r)}(f_1^{(l)})^{-1}(c_i^{(1)})^{-1}df_1^{(l)}(f_2^{(r)})^{-1}f_2^{(k)}(f_1^{(l)})^{-1}d^{-1}c_i^{(1)}f_1^{(l)}(f_1^{(r)})^{-1}f_2^{(l)}(f_2^{(r)})^{-1}\\
 =&f_2^{(r)}(f_2^{(l)})^{-1}f_1^{(r)}(f_1^{(l)})^{-1}(c_i^{(1)})^{-1}d(f_2^{(r)})^{-1}df_1^{(l)}(f_1^{(l)})^{-1}d^{-1}f_2^{(k)}d^{-1}c_i^{(1)}f_1^{(l)}(f_1^{(r)})^{-1}f_2^{(l)}(f_2^{(r)})^{-1}\\
 =&f_2^{(r)}(f_2^{(l)})^{-1}f_1^{(r)}(f_1^{(l)})^{-1}(c_i^{(1)})^{-1}d(f_2^{(r)})^{-1}f_2^{(k)}d^{-1}c_i^{(1)}f_1^{(l)}(f_1^{(r)})^{-1}f_2^{(l)}(f_2^{(r)})^{-1}\\
  =& f_2^{(r)}(f_2^{(l)})^{-1}f_1^{(r)}(f_2^{(r)})^{-1}f_2^{(k)}(f_1^{(r)})^{-1}f_2^{(l)}(f_2^{(r)})^{-1}\\
  =& f_2^{(r)}f_1^{(r)}(f_2^{(r)})^{-1}f_2^{(k)}(f_1^{(r)})^{-1}(f_2^{(r)})^{-1}\\
\end{split}\]

where the last equality follows from the vanishing
\[\begin{split}
  \mbox{ } &(f_2^{(l)})^{-1}f_1^{(r)}(f_2^{(r)})^{-1}f_2^{(k)}(f_1^{(r)})^{-1}f_2^{(l)}f_1^{(r)}(f_2^{(k)})^{-1}f_2^{(r)}(f_1^{(r)})^{-1}\\
  &\sim \left[f_2^{(l)},(f_1^{(r)})^{-1}\right](f_2^{(r)})^{-1}f_2^{(k)}\left[(f_1^{(r)})^{-1},f_2^{(l)}\right](f_2^{(k)})^{-1}f_2^{(r)}\\
  & \sim (f_2^{(k)})^{-1}(f_2^{(r)})^{-1}d f_2^{(r)}f_2^{(k)}(f_2^{(r)})^{-1}(f_2^{(k)})^{-1}d^{-1} f_2^{(k)}f_2^{(r)}=1
\end{split}\]

Using $\left[\left[x_1,x_2\right],d\right]$ we show that relation (3) follows from $\mathcal{M}^C$. We prove the case $i=1$, the case $i=2$ is similar.

\[\begin{split}
  &\left[\left[x_2,x_1\right]d^{-1},f_1^{(k)}\right]\\
  &= f_2^{(r)}f_1^{(r)}(f_2^{(r)})^{-1}(f_1^{(r)})^{-1}d^{-1}f_1^{(k)}df_1^{(r)}f_2^{(r)}(f_1^{(r)})^{-1}(f_2^{(r)})^{-1}(f_1^{(k)})^{-1}\\
  &\sim \left((f_2^{(r)})^{-1}(f_1^{(k)})^{-1}d^{-1}f_2^{(r)}f_1^{(k)}\right)(f_1^{(k)})^{-1}f_1^{(r)}(f_2^{(r)})^{-1}(f_1^{(r)})^{-1}f_1^{(k)}df_1^{(r)}f_2^{(r)}(f_1^{(r)})^{-1}\\
  &=f_1^{(r)}(f_1^{(k)})^{-1}(f_2^{(r)})^{-1}f_1^{(k)}(f_1^{(r)})^{-1}df_1^{(r)}f_2^{(r)}(f_1^{(r)})^{-1}\\
  &=f_1^{(r)}(f_2^{(r)})^{-1}(f_1^{(k)})^{-1}d^{-1}f_1^{(k)}(f_1^{(r)})^{-1}df_1^{(r)}f_2^{(r)}(f_1^{(r)})^{-1}\\
  &=f_1^{(r)}(f_2^{(r)})^{-1}(f_1^{(r)})^{-1}d^{-1}f_1^{(r)}(f_1^{(r)})^{-1}df_1^{(r)}f_2^{(r)}(f_1^{(r)})^{-1}=1
\end{split}\]

We finish the proof by showing that relation (4) follows from the relations in $\mathcal{M}^C$. This is equivalent to proving that modulo $\mathcal{M}^C$ the equality
\[
 g_i^{(k)}f_2^{(r)}f_1^{(r)}(f_2^{(r)})^{-1}f_2^{(l)}(f_1^{(r)})^{-1}(f_2^{(l)})^{-1}(g_i^{(k)})^{-1}=f_2^{(r)}f_1^{(r)}(f_2^{(r)})^{-1}f_2^{(l)}(f_1^{(r)})^{-1}(f_2^{(l)})^{-1}
\]
holds.

Using $\left[f_i^{(l)},g_j^{(k)}\right]W_{i,j,1}^{-1}=1$ we obtain that
\[\begin{split}
  &g_i^{(k)}f_2^{(r)}f_1^{(r)}(f_2^{(r)})^{-1}f_2^{(l)}(f_1^{(r)})^{-1}(f_2^{(l)})^{-1}(g_i^{(k)})^{-1}\\
  =& W_{2,i,1}^{-1}f_2^{(r)}W_{1,i,1}^{-1}f_1^{(r)}(f_2^{(r)})^{-1}f_2^{(l)}(f_1^{(r)})^{-1}W_{1,i,1}(f_2^{(l)})^{-1} W_{2,i,1}
\end{split}\]

As in (3) we distinguish the cases $i=1,2$. Again we will only show how to prove the case $i=1$, the case $i=2$ being similar:
\[\begin{split}
  &W_{2,1,1}^{-1}f_2^{(r)}W_{1,1,1}^{-1}f_1^{(r)}(f_2^{(r)})^{-1}f_2^{(l)}(f_1^{(r)})^{-1}W_{1,1,1}(f_2^{(l)})^{-1} W_{2,1,1}\\
  =& W_{2,1,1}^{-1}f_2^{(r)}(c_1^{(1)})^{-1}f_1^{(r)}c_1^{(1)}(f_1^{(r)})^{-1}f_1^{(r)}(f_2^{(r)})^{-1}f_2^{(l)}(f_1^{(r)})^{-1}f_1^{(r)}(c_1^{(1)})^{-1}(f_1^{(r)})^{-1}c_1^{(1)}(f_2^{(l)})^{-1} W_{2,1,1}\\
  =& W_{2,1,1}^{-1}f_2^{(r)}f_1^{(r)}(f_2^{(r)})^{-1}f_2^{(l)}(f_1^{(r)})^{-1}(f_2^{(l)})^{-1} W_{2,1,1}\\
  =& (c_1^{(1)})^{-1}df_2^{(r)}c_1^{(1)}(f_2^{(r)})^{-1}f_2^{(r)}f_1^{(r)}(f_2^{(r)})^{-1}f_2^{(l)}(f_1^{(r)})^{-1}(f_2^{(l)})^{-1}f_2^{(l)}(c_1^{(1)})^{-1}(f_2^{(l)})^{-1}d^{-1}c_1^{(1)}\\
  =& f_2^{(r)}f_1^{(r)}(f_2^{(r)})^{-1}f_2^{(l)}(f_1^{(r)})^{-1}(f_2^{(l)})^{-1}
\end{split}\]
where in the last equality we use $\left[\delta,c_i^{(1)}\right]=1$ and $\left[d,\left[x_1,x_2\right]\right]=1$. Hence, relation (4) indeed follows from the relations $\mathcal{M}^C$, completing the proof.
\end{proof}

\begin{proof}[Proof of Lemma \ref{lemProofMT3}]
 First observe that using $\delta^{-1}=\left[x_1,x_2\right] d$,\\ $\left[\left[x_1,x_2\right],d\right]$, $df_1^{(r)}f_2^{(r)}=f_1^{(r)}(f_1^{(l)})^{-1}df_1^{(l)}f_2^{(r)}$ and $df_1^{(l)}f_2^{(r)}=f_2^{(r)}f_1^{(l)}$, we obtain
\[\begin{split}
  x_2^{-1}x_1^{-1}\delta^{-2}x_2x_1&= \left\{(f_1^{(r)})^{-1}(f_2^{(r)})^{-1}df_1^{(r)}f_2^{(r)}\right\}(f_2^{(r)})^{-1}(f_1^{(l)})^{-1}f_2^{(r)}f_1^{(l)}.\\
\end{split}\]

Using that $\left[c_1^{(1)}g_1^{(l)},c_2^{(1)}g_2^{(r)}\right]d=1$ we obtain
\[\begin{split}
 &\left[(c_1^{(1)}g_1^{(r)})^{-1},(c_1^{(1)}g_2^{(r)})^{-1}\right]\\
 =&(c_1^{(1)}g_1^{(r)})^{-1}(c_2^{(1)}g_2^{(r)})^{-1}(c_1^{(1)}g_1^{(l)})(c_2^{(1)}g_2^{(r)})(c_1^{(1)}g_1^{(l)})^{-1}(c_1^{(1)}g_1^{(r)})\\ & \cdot \left\{(c_1^{(1)}g_1^{(r)})^{-1}(c_2^{(1)}g_2^{(r)})^{-1}d(c_1^{(1)}g_1^{(r)})(c_2^{(1)}g_2^{(r)})\right\}\\
\end{split}\]

But this means that the equivalence of the two relations is equivalent to
\[
 \left[(f_2^{(r)})^{-1},(f_1^{(l)})^{-1}\right](c_1^{(1)}g_1^{(r)})^{-1}\left[(c_2^{(1)}g_2^{(r)})^{-1},(c_1^{(1)}g_1^{(l)})\right](c_1^{(1)}g_1^{(r)})=1
\]
modulo all other relations.

We will prove that this term does indeed vanish
\[\begin{split}
  &\left[(f_2^{(r)})^{-1},(f_1^{(l)})^{-1}\right](c_1^{(1)}g_1^{(r)})^{-1}\left[(c_2^{(1)}g_2^{(r)})^{-1},(c_1^{(1)}g_1^{(l)})\right](c_1^{(1)}g_1^{(r)})\\
  &=(f_2^{(r)})^{-1}(f_1^{(l)})^{-1}df_1^{(l)}f_2^{(r)}(c_1^{(1)}g_1^{(r)})^{-1}(c_2^{(1)}g_2^{(r)})^{-1}d^{-1}(c_2^{(1)}g_2^{(r)})(c_1^{(1)}g_1^{(r)})\\
  &\stackrel{c_i^{(1)}g_i^{(r)}=f_i^{(r)}c_i^{(r)}}{=} (f_2^{(r)})^{-1}(f_1^{(l)})^{-1}df_1^{(l)}f_2^{(r)}(c_1^{(1)}g_1^{(r)})^{-1}(c_2^{(r)})^{-1}(f_2^{(r)})^{-1}d^{-1}f_2^{(r)}c_2^{(r)}(c_1^{(1)}g_1^{(r)})\\
\end{split}\]

Finally the relations $\left[c_j^{(r)},d\right]=1$, $\left[c_j^{(r)},f_j^{(l)}\right]=1$, $(f_i^{(r)})^{-1}d^{-1}f_i^{(r)}=(f_i^{(k)})^{-1}d^{-1}f_i^{(k)}$,\\ $(f_1^{(l)})^{-1}(f_2^{(k)})^{-1}d^{-1}f_2^{(k)}f_1^{(l)}=(f_2^{(k)})^{-1}(f_1^{(l)})^{-1}d^{-1}f_1^{(l)}f_2^{(k)}$, and $c_1^{(1)}g_1^{(r)}=f_1^{(r)}c_1^{(r)}$ imply
\[\begin{split}
 &(f_2^{(r)})^{-1}(f_1^{(l)})^{-1}df_1^{(l)}f_2^{(r)}(c_1^{(1)}g_1^{(r)})^{-1}(c_2^{(r)})^{-1}(f_2^{(r)})^{-1}d^{-1}f_2^{(r)}c_2^{(r)}(c_1^{(1)}g_1^{(r)})\\
 &=(f_2^{(r)})^{-1}(f_1^{(l)})^{-1}df_1^{(l)}f_2^{(r)}(f_1^{(l)})^{-1}(f_2^{(r)})^{-1}d^{-1}f_2^{(r)}f_1^{(l)}\sim ddd^{-1}d^{-1}=1 
\end{split}\]

Hence, we can replace the relation $x_2^{-1}x_1^{-1}\delta^{-2}x_2x_1\left[(x_1c_1^{(r)})^{-1},(x_2c_2^{(r)})^{-1}\right]=1$ by the relation $S^{(r)}\cdot T^{(r)}=1$ using Tietze transformations.
\end{proof}

\bibliography{References}
\bibliographystyle{amsplain}

\end{document}